\documentclass[12pt,a4paper]{amsart}

\usepackage{amssymb,amsfonts,amsthm,a4,amsmath}

\newtheorem{thm}{Theorem}[section]
\newtheorem{cor}[thm]{Corollary}
\newtheorem{lem}[thm]{Lemma}
\newtheorem{clai}[thm]{Claim}
\newtheorem{prop}[thm]{Proposition}
\theoremstyle{definition}
\newtheorem{defn}[thm]{Definition}
\theoremstyle{remark}
\newtheorem{rem}[thm]{Remark}
\numberwithin{equation}{section}

\newcommand{\eps}{\varepsilon}
\newcommand{\Cone}{\textnormal{Cone}}
\newcommand{\Precone}{\textnormal{Precone}}
\newcommand{\SL}{\textnormal{SL}}
\newcommand{\F}{\mathbf{F}}
\newcommand{\scu}{\textnormal{scu}}
\newcommand{\wid}{\textnormal{wid}}

\title{Dehn function and asymptotic cones of Abels' group}
\author{Yves Cornulier, Romain Tessera}
\date{March 3, 2013}
\subjclass[2010]{Primary 20F65; Secondary 20F69, 20F16, 20F05, 22D05, 51M25}
\thanks{Y.~C.~was partly supported by A.N.R. project GSG 12-{\sc BS}01-0003-01; R.~T.~was supported by A.N.R.\ project AGORA ({\sc ANR-09-BLAN-0059}) and A.N.R.\ project GGAA ({\sc ANR-10-BLAN-0116})}

\begin{document}

\baselineskip=16pt

\maketitle

\begin{abstract}
We prove that Abels' group over an arbitrary nondiscrete locally compact field has a quadratic Dehn function. As applications, we exhibit connected Lie groups and polycyclic groups whose asymptotic cones have uncountable abelian fundamental group. We also obtain, from the case of finite characteristic, uncountably many non-quasi-isometric finitely generated solvable groups, as well as peculiar examples of fundamental groups of asymptotic cones. 
\end{abstract}

\section{Introduction}

Let $R$ be a commutative ring. We consider the following solvable group, introduced by Abels \cite{Ab1}.

\begin{equation}A_4(R)=\left\{\left(\begin{array}{cccccc}
1 & x_{12} & x_{13} & x_{14}\\
0 &  t_{22} & x_{23} & x_{24}\\
0 & 0 & t_{33}  & x_{34} \\
0 & 0 & 0 & 1
\end{array}\right)\label{abeq}:\;x_{ij}\in R;\;t_{ii}\in R^\times.\right\}\end{equation}
Observe that if we denote by $Z(R)$ the subgroup generated by unipotent matrices whose only nonzero off-diagonal entry is $x_{14}$, then $Z(R)$ is central in $A_4(R)$. Abels' initial motivation was to exhibit a finitely presented group with an infinitely generated central subgroup, namely $A_4(\mathbf{Z}[1/p])$. In particular, its quotient by the central subgroup $Z(\mathbf{Z}[1/p])$ is not finitely presented, showing that the class of finitely presented solvable groups is not stable under taking quotients.

\subsection*{Abels' group over locally compact fields}

Finite presentability of $A_4(\mathbf{Z}[1/p])$ is closely related to the fact that $A_4(\mathbf{Q}_p)$  is compactly presented \cite{Ab2}, motivating the study of Abels' group over nondiscrete locally compact fields. 
Our main goal in this paper is to provide the following quantitative version of this result.

\begin{thm}\label{ak}
For every nondiscrete locally compact field $\mathbf{K}$, the group $A_4(\mathbf{K})$ has a quadratic Dehn function.
\end{thm}

This means that loops of length $r$ in $A_4(\mathbf{K})$ can be filled with area in $O(r^2)$ when $r\to\infty$. This extends Abels' result \cite{Ab1,Ab2} that $A_4(\mathbf{K})$ is compactly presented if $\mathbf{K}$ has characteristic zero; besides, Abels's result is nontrivial only when $\mathbf{K}$ is ultrametric, while Theorem \ref{ak} is meaningful when $\mathbf{K}$ is Archimedean too. 

Recall that given a metric space $(X,d)$ and a nonprincipal ultrafilter $\omega$ on the set of positive integers, a certain metric space $\Cone_\omega(X)$ can be defined as an ``ultralimit" of the metric spaces $(X,\frac1n d)$ when $n\to\omega$, and is called the asymptotic cone of $X$ along $\omega$. (The precise definitions will be recalled in Section \ref{recall}.). The bilipschitz type of $\Cone_\omega(X)$ is a quasi-isometry invariant of $X$. 

By Papasoglu's theorem \cite{Pap}, a quadratic Dehn function implies that every asymptotic cone is simply connected, so we obtain

\begin{cor}\label{csco}
For every nondiscrete locally compact field $\mathbf{K}$ and every nonprincipal ultrafilter $\omega$, the asymptotic cone $\Cone_\omega(A_4(\mathbf{K}))$ is simply connected.
\end{cor}

\subsection*{Asymptotic cones and central extensions}

Corollary \ref{csco} can be used to obtain various examples of unusual asymptotic cones, using generalities on central extensions, which we now partly describe (see Theorem \ref{pico} for a full version). 
Our main tool relates, for certain central extensions
\begin{equation}\label{zgq}1\to Z\to G\to Q\to 1,\end{equation}
the fundamental group $\pi_1(\Cone_\omega(Q))$ and the group $\Cone_\omega(Z)$, where $Z$ is endowed with the restriction of the metric of $G$. It can be viewed as an analogue of the connection between the fundamental group of a Lie group and the center of its universal cover. Notably, it implies the following

\begin{thm}[Corollary \ref{conab}]\label{gq}
Given a central extension as in (\ref{zgq}), if $G,Q$ are compactly generated locally compact groups, $\Cone_\omega(G)$ is simply connected and $\Cone_\omega(Z)$ is ultrametric, then $$\pi_1(\Cone_\omega(Q))\simeq\Cone_\omega(Z).$$
\end{thm}

As a direct application of Corollary \ref{csco} and Theorem \ref{gq}, we deduce

\begin{cor}[Corollary \ref{comoc}]\label{aba}
For every nondiscrete locally compact field $\mathbf{K}$ and nonprincipal ultrafilter $\omega$, the group $\pi_1(\Cone_\omega(A_4(\mathbf{K})/Z(\mathbf{K})))$ is an abelian group with continuum cardinality.
\end{cor}

\begin{rem}
If $\mathbf{K}$ is non-Archimedean, $Z(\mathbf{K})$ is not compactly generated, and it follows that $A_4(\mathbf{K})/Z(\mathbf{K})$ is not compactly presented. Similarly, if $\mathbf{K}$ is Archi\-medean, the exponential distortion of $Z(\mathbf{K})$ implies that the Lie group $A_4(\mathbf{K})/Z(\mathbf{K})$ has an exponential Dehn function.
\end{rem}

\subsection*{Application to discrete groups}

We further obtain, from a slight variant of Theorem \ref{ak}, results concerning discrete groups. The first corollary, proved in Section \ref{ewl}, is the following

\begin{cor}\label{cpoly}
There exists a polycyclic group $\Lambda$, namely $A_4(R)/Z(R)$ if $R$ is the ring of integers of a totally real number field of degree 3, such that $\Lambda$ has an exponential Dehn function, and for every $\omega$, the fundamental group $\pi_1(\Cone_\omega(\Lambda))$ is abelian and nontrivial, namely isomorphic to a $\mathbf{Q}$-vector space of continuum dimension.
\end{cor}

The structure of $\pi_1(\Cone_\omega(\Lambda))$ as a {\em topological group} is given in Proposition \ref{car0c}.
In the case of characteristic $p$, we prove

\begin{cor}[Theorem \ref{a4r}]\label{a4p}
The group $A_4(\F_p[t,t^{-1},(t-1)^{-1}])$ is finitely presented and has a quadratic Dehn function.
\end{cor}
The finite presentation of this group seems to be a new result. Note that $A_4(\F_p[t])$ is not finitely generated, while $A_4(\F_p[t,t^{-1}])$ is finitely generated but not finitely presented (see Remark \ref{a4notfg}).
There were previous (substantially more complicated) examples of finitely presented solvable groups whose center contains an infinite-dimensional $\F_p$-vector space in \cite[\S 2.4]{BGS} and \cite[\S 2]{Kha}; those examples have an undecidable word problem, so they are not residually finite and their Dehn functions are not recursively bounded. Corollary \ref{a4p}, again combined with Theorem \ref{gq}, has the following three corollaries. The first proof of the existence of uncountably many non-quasi-isometric finitely generated groups was due to Grigorchuk \cite{Gri}, based on the word growth. Later, Bowditch \cite{Bo} used methods of small cancelation to construct uncountably many non-quasi-isometric finitely generated non-amenable groups. His approach also involves filling of loops, although not in the asymptotic cone.

\begin{cor}\label{contifg}
There exist continuum many pairwise non-quasi-isometric solvable (actually (3-nilpotent)-by-abelian) finitely generated groups. Namely, for each prime $p$, such groups can be obtained as quotients of $A_4(\F_p[t,t^{-1},(t-1)^{-1}])$ by central subgroups.
\end{cor}

To distinguish this many classes, we associate, to any metric space $X$, the subset of the set $\mathcal{U}_\infty(\mathbf{N})$ of nonprincipal ultrafilters on the integers
$$\nu(X)=\{\omega\in\mathcal{U}_\infty(\mathbf{N}):\;\Cone_\omega(X)\textnormal{ is simply connected}\};$$
the subset $\nu(X)\subset\mathcal{U}_\infty(\mathbf{N})$ is a quasi-isometry invariant of $X$, and we obtain the corollary by proving that $\nu$ achieves continuum many values on a certain class of groups (however, the subset $\nu(X)$ can probably not be arbitrary).

Corollary \ref{contifg} is proved in \S\ref{fpj}, as well as the following one, which relies on similar ideas.

\begin{cor}\label{littlepi}
There exists a (3-nilpotent)-by-abelian finitely generated group $R$, namely a suitable central quotient of $A_4(\mathbf{F}_p[t,t^{-1},(t-1)^{-1}])$, for which for every nonprincipal ultrafilter $\omega$, we have $\pi_1(\Cone_\omega(R))$ is isomorphic to either $\F_p$ (cyclic group on $p$ elements) or is trivial, and is isomorphic to $\F_p$ for at least one nonprincipal ultrafilter.
\end{cor}

To put Corollary \ref{littlepi} into perspective, recall that Gromov \cite[2.$\textnormal{B}_1$(d)]{Gro} asked whether an asymptotic cone of a finitely generated group always has trivial or infinitely generated fundamental group. In \cite{OOS}, a first counterexample was given, for which for some nonprincipal ultrafilter (in their language: for a fixed --quickly growing-- scaling sequence and all ultrafilters) the fundamental group is $\mathbf{Z}$. Here we provide 
the first example for which the fundamental group is finite and nontrivial.
Moreover we use the scaling sequence $(1/n)$, which reads all scales, so in our example the fundamental group is finite for all nonprincipal ultrafilters and scaling sequences. 

The first example of a finitely generated group with two non-homeomorphic asymptotic cones was obtained by Thomas and Velickovic \cite{TV}, and was improved to $2^{\aleph_0}$ non-homeomorphic asymptotic cones by Drutu and Sapir \cite{DS}. These examples are not solvable groups, although amenable examples (satisfying no group law) appear in \cite{OOS}. Corollary \ref{littlepi} provides the first examples of finitely generated solvable groups with two non-homeomorphic asymptotic cones. The next corollary, obtained in \S\ref{fpj} by a variation on the proof of the previous one, improves it to infinitely many non-homeomorphic asymptotic cones.

\begin{cor}\label{resomany}
There exist a (3-nilpotent)-by-abelian finitely generated group, namely a suitable central quotient of $A_4(\mathbf{F}_p[t,t^{-1},(t-1)^{-1}])$, with at least $2^{\aleph_0}$ non-bilipschitz homeomorphic, respectively at least $\aleph_0$ non-homeomorphic, asymptotic cones.
\end{cor}

The fundamental group of a metric space has a natural bi-invariant pseudo-metric space structure, whose bilipschitz class is a bilipschitz invariant of the metric space, and the isomorphism in Theorem \ref{gq} is actually bilipschitz. This extra-feature is used in the proof of the bilipschitz statement of Corollary \ref{resomany}. Actually, this pseudo-metric structure on the fundamental group of the cone is used in a crucial way in the proof of Theorem \ref{gq} itself.

In view of Theorem \ref{gq}, Corollaries \ref{contifg}, \ref{littlepi}, and \ref{resomany} are all based on a systematic study in Section \ref{csds} of the asymptotic cone of the infinite direct sum $\F_p^{(\mathbf{N})}$, endowed with a suitable invariant ultrametric. In particular, Theorem \ref{thm:cla} provides a complete classification of these asymptotic cones up to isomorphism of topological groups.

\subsection*{Organization}

Section \ref{recall} recalls the definition of Dehn function and asymptotic cone. Section \ref{ak} is devoted to the proof of Theorem \ref{ak}. In order to carry over the results to discrete groups, a minor variant of Theorem \ref{ak} is proved in Section \ref{ewl}. Finally Section \ref{csds} deals with asymptotic cones of the metric group $\F_p^{(\mathbf{N})}$, whose description leads to the latter corollaries.

\medskip

\noindent {\bf Acknowledgments.}  The main results of the paper were obtained when Y.C.\ was visiting R.T.\ in Vanderbilt University. We thank people in the Math Department in Vanderbilt University,
especially Guoliang Yu, for their hospitality. We thank Denis Osin and Mark Sapir for interesting discussions about asymptotic cones. We are indebted to the referee for an extremely careful job.

\setcounter{tocdepth}{1}
\tableofcontents

\section{Dehn function, asymptotic cones}\label{recall}

\subsection{Dehn function}

Let $G$ be a locally compact group, generated by some compact symmetric subset $S$. It is {\em compactly presented} if for some $r$, the set $\mathcal{R}_r$ of relations of length at most $r$ between elements of $S$ generates the set of all relations, i.e.\ $G$ admits the presentation $\langle S|\mathcal{R}_r\rangle$.

If this holds, and if $w$ is a relation, i.e.\ an element of the kernel $K_S$ of the natural homomorphism $F_S\to G$ (where $F_S$ is the free group over $S$), the area $\alpha(w)$ of $w$ is defined as the length of $w$ with respect to the union of $F_S$-conjugates of $\mathcal{R}_r$. The Dehn function is then defined as
$$\delta(n)=\sup\{\alpha(w):\;w\in K_S,\,|w|_S\le n\}.$$
This function takes finite values. Although it depends on $S$ and $r$, its asymptotic behavior only depends on $G$. By convention, if $G$ is not compactly presented, we say it has an identically infinite Dehn function. Any two quasi-isometric locally compact compactly generated groups have asymptotically equivalent Dehn functions.

\subsection{Asymptotic cone}

Let $(X,d)$ be a metric space and $\omega$ a nonprincipal ultrafilter on the set $\mathbf{N}$ of positive integers. Define 
$$\Precone(X)=\{x=(x_n):x_n\in X,\limsup d(x_n,x_0)/n<\infty\}.$$
Define the pseudo-distance $$d_\omega(x,y)=\lim_\omega\frac1n d(x_n,y_n);$$
and $\Cone_\omega(X)$ as the Hausdorffication of $\Precone(X)$, endowed with the resulting distance $d_\omega$.

If $X=G$ is a group and $d$ is left-invariant, then $\Precone(G)$ (written $\Precone(G,d)$ if needed) is obviously a group, $d_\omega$ is a left-invariant pseudodistance, and $\Cone_\omega(G)=\Precone(G)/H$, where $H$ is the subgroup of elements at pseudodistance 0 from the identity. It is not normal in general. However in case $d$ is bi-invariant (e.g.\ when $G$ is abelian), it is normal and $\Cone_\omega(G,d)$ is then naturally a group endowed with a bi-invariant distance.

Taking asymptotic cones is a functor between metric spaces, mapping large-scale Lipschitz maps to Lipschitz maps and identifying maps at bounded distance. In particular, it maps quasi-isometries to bilipschitz homeomorphisms.

\section{Proof of Theorem \ref{ak}}\label{pak}

Fix a nondiscrete locally compact field $\mathbf{K}$, fix an absolute value $|\cdot|$ on $\mathbf{K}$, and write $G=A_4(\mathbf{K})$; let us show that $G$ has a quadratic Dehn function. The proof below partly uses some arguments borrowed from the metabelian case \cite{CT}, which apply to several subgroups of $G$ (see Lemma \ref{quadratic_subgroup_claim}), as well as Gromov's trick (see \S\ref{gtrick}). However, the presence of a distorted center entails bringing in several new arguments, gathered in \S\ref{ppq}.

\subsection{Generation of $G$}

Fix $t\in\mathbf{K}$ with $|t|>1$. Let $G\subset A_4(\mathbf{K})$ be the subgroup of elements whose diagonal $(1,t_{22},t_{33},1)$ consists of elements of the form $(1,t^{n_2},t^{n_3},1)$ with $(n_2,n_3)\in \mathbf{Z}^2$. So $G$ is closed and cocompact in $A_4(\mathbf{K})$ and we shall therefore stick to $G$.

Let $D$ be the set of diagonal elements in $G$ and let $T$ be the set of diagonal elements as above with $(n_1,n_2)\in\{(\pm 1,0),(0,\pm 1)\}$. Let $U_{ij}$ be the set of upper unipotent elements with all upper coefficients vanishing except $u_{ij}$, and $U_{ij}^1$ the set of elements in $U_{ij}$ with $|u_{ij}|\le 1$.

Define $S=T\cup W$, where $W=\bigcup_{1\le i<j\le 4,(i,j)\neq (1,4)}U_{ij}^1$. This is a compact, symmetric subset of $G$. We begin by the easy

\begin{lem}\label{gener}
The subset $S$ generates $G$. 
\end{lem}
\begin{proof}
Observe that $T$ generates $D$, and that for all $(i,j)$ with $i<j$ and $(i,j)\neq (1,4)$, $D$ and $U_{ij}^1$ generates $DU_{ij}$. Moreover $U_{14}\subset [U_{12},U_{24}]$. So the subgroup generated by $S$ contains $D$ and $U_{ij}$ for all $i<j$. These clearly generate $G$.
\end{proof}

Actually, we need a more quantified statement. Let $[T]$ be the set of words in $T$, and $|\cdot|_T$ the word length in $[T]$.
\begin{lem}\label{generq}
There exists a constant $C>0$ such that every element in the $n$-ball $S^n$ can be written as 
\begin{equation}\label{peig}d\prod_{i=1}^9t_is_it_i^{-1},\end{equation} where $d,t_i\in [T]$, $s_i\in W$, and $|c|_T\le n$, $|t_i|_T|\le Cn$. 
\end{lem}
\begin{proof}
Start from an element $g$ in $G$. It can be written in a unique way $du_{12}u_{23}u_{34}u_{13}u_{24}u_{14}$ with $d\in D$, $u_{ij}\in U_{ij}$. In turn, $u_{14}$ can be uniquely written as $[v_{12},v_{24}]$, where the $(1,2)$-entry of $v_{12}\in U_{12}$ is 1 and $v_{24}\in U_{24}$.

If we assume that $|g|_T\le n$, then clearly $|d|_T\le n$, and there exists a universal constant $C>1$ (only depending on the absolute value of $t$) such that the nonzero upper-diagonal entry of the $u_{ij}$ or $v_{24}$ is at most $C^n$. So each $u_{ij}$ can be written as $\gamma_{ij}s_{ij}\gamma_{ij}^{-1}$, where $\gamma_{ij}\in [T]$ has length $\le C'n$ an $s_{ij}\in U_{ij}^1$. Similarly $v_{24}$ can be written this way. (Actually, $C'=1$ works if $\mathbf{K}$ is ultrametric.)
\end{proof}

\subsection{Subgroups of $G$ with contracting elements}

An easy way to prove quadratic filling is the use of elements whose action by conjugation on the unipotent part is contracting. 
Although $G$ itself does not contain such elements, we will show that it contains large enough such subgroups.  More precisely, the proof that $G$ has a quadratic Dehn function (implicitely) consists in showing that $G$ is an amalgamated product of finitely many subgroups containing contractions (Abels used a similar strategy to show that $G$ is compactly presented). 

\begin{lem}\label{quadratic_subgroup_claim}
Let $G_1$ (resp.\ $G_2$, resp.\ $G_3$, resp.\ $G_4$) be the subgroup of matrices
in $G$ such that $x_{14}=x_{24}=x_{34}=0$ (resp.\
$x_{12}=x_{13}=x_{14}=0$, resp.\ $x_{12}=x_{23}=x_{34}=x_{14}=0$, resp.\ $x_{13}=x_{14}=x_{23}=x_{24}=0$).
Then all $G_i$ have quadratic Dehn functions.
\end{lem}
\begin{proof}
All the verifications are based on a standard contraction argument. Let us start with $G_1$. Observe that in $G_1$, the left conjugation by the diagonal matrix $q=(1,t,t^2,1)$ maps $S$ into itself, so is 1-Lipschitz on $(G,d_S)$. Thus the right multiplication by $q^{-1}$ is 1-Lipschitz as well; it is moreover of bounded displacement (as any right multiplication), actually 3. Moreover it is contractive on the unipotent part. Since the 9-ball has nonempty interior, there exists a constant $C$ such that for any $n$, if $B(n)$ is the $n$-ball around 1 in $G$, for every $g\in B(n)$, $gq^{Cn}$ is at distance at most 9 from its projection to $D$.

So starting from a loop of size $n$, $Cn$ successive right multiplications by $q$ homotope it to a loop of the same size, in which each element is at distance at most 9 from its projection to $D$. Each of these multiplications has cost $\le 3n$, provided that for every $s\in S$, the relator $qsq^{-1}s'^{-1}$ is in $\mathcal{R}$, where $s'$ is the unique element of $S$ represented by $qsq^{-1}$ (which holds if all relations of length 8 are relators). This loop can be homotoped to its projection to $D$, provided all relations of size 20=9+1+9+1 are relators. Finally this loop in $D$ has quadratic area (provided the obvious commuting relator of $D\simeq\mathbf{Z}^2$ is a relator). 

By symmetry, $G_2$ is similar.

Finally, $G_3$ is also similar, using instead of $q$ the diagonal element $q'=(1,t^{-1},t,1)$, whose left conjugation contracts both $U_{13}$ and $U_{24}$, and $q'^{-1}$ works for $G_4$.
\end{proof}

\subsection{Gromov's trick}\label{gtrick}

Let $\mathcal{R}(C,n)$ be the set of null-homotopic words of the form
$$\prod_{i=1}^{27}t_is_it_i^{-1},$$
where $t_i\in [T]$, $|t_i|\le Cn$, and $s_i\in W$.

\begin{prop}\label{quadrapei}Let $M$ be large enough, and define the set of relators as the set of all words in $S$ of length $\le M$, that represent the trivial element of $G$.
For every $C$, there exists $C'$ such that every word in $\mathcal{R}(C,n)$ has area $\le C'n^2$.
\end{prop}

By Lemmas \ref{gener} and \ref{generq}, as well as Gromov's trick \cite[Prop.\ 4.3]{CT}, to show that $G$ has a quadratic Dehn function, it is enough to show that null-homotopic words that are concatenation of three words of the form (\ref{peig}) have quadratic area. By an obvious conjugation and using that $\mathbf{Z}^2$ has a quadratic Dehn function, this follows from Proposition \ref{quadrapei}, which we now proceed to prove.

\subsection{Proof of Proposition \ref{quadrapei}}\label{ppq}

To simplify the exposition, we will adopt the following convenient
language: the phrase ``the word $w$ can be
replaced by the word $w'$, with a quadratic cost", means that
there are two universal constants $C_1,C_2$ such that
$\ell(w')\leq C_1\ell(w)$, and $\alpha(w^{-1}w')\leq C_2\ell(w)^2$. This
will be denoted, for short by: $w\rightsquigarrow_2 w'$.

Let $J=\{(i,j), 1\leq i<j\leq 4, (i,j)\neq (1,4)\}.$ 
For $(i,j)\in J$, let $e_{ij}^x$ be the elementary matrix with entry $(i,j)$ equal to $x$, and fix a word $\hat{e}_{ij}^x$ of minimal length in $T\cup U_{ij}^1$ representing $e_{ij}^x$. 

Using Lemma~\ref{quadratic_subgroup_claim} (or the even easier observation that $DU_{ij}$ has a quadratic Dehn function for all $(i,j)$), Proposition \ref{quadrapei} reduces to proving that, given $c>1$, words of the form
\begin{equation}w=\prod_{k=1}^{27}\hat{e}_{i_kj_k}^{x_k}\quad(\sup_k|x_k|\le c^n)\label{poum}\end{equation}
have area in $O_c(n^2)$.

Lemma~\ref{quadratic_subgroup_claim} has the following immediate consequence.

\begin{clai}\label{clai23}
We have
\begin{enumerate}
\item $[\hat{e}_{12}^{u},\hat{e}_{23}^{v}]\rightsquigarrow_2
\hat{e}_{13}^{uv}$ and $\hat{e}_{13}^{u}\rightsquigarrow_2
[\hat{e}_{12}^{1},\hat{e}_{23}^{u}]$ ;
\item
$[\hat{e}_{23}^{u},\hat{e}_{34}^{v}]\rightsquigarrow_2 \hat{e}_{24}^{uv}$ and $\hat{e}_{24}^{u}\rightsquigarrow_2
[\hat{e}_{23}^{u},\hat{e}_{34}^{1}]$;
\item $[\hat{e}_{23}^{u},\hat{e}_{24}^{v}]\rightsquigarrow_2 1$;
\item $[\hat{e}_{23}^{u},\hat{e}_{13}^{v}]\rightsquigarrow_2 1$.
\end{enumerate}
\end{clai}

\begin{clai}\label{clai2}
We have
\begin{enumerate}
\item\label{2vii} for any $(i,j)\in J$,
$\hat{e}_{i j}^u\hat{e}_{i j}^v\rightsquigarrow_2 \hat{e}_{i j}^{u+v}$ and $(\hat{e}_{ij}^{u})^{-1}\rightsquigarrow_2 \hat{e}_{ij}^{-u}$;
\item $[\hat{e}_{1 2}^{u}, \hat{e}_{3 4}^{v}]\rightsquigarrow_2 1$;
\item\label{2iv} $[\hat{e}_{1 3}^{u},\hat{e}_{2 4}^{v}]\rightsquigarrow_2 1$;
\item  $[\hat{e}_{3 4}^{u},\hat{e}_{2 4}^{v}]\rightsquigarrow_2 1$;
\item $[\hat{e}_{1 2}^{u},\hat{e}_{1 3}^{v}]\rightsquigarrow_2 1$.
\end{enumerate}
\end{clai}

\begin{clai}\label{clai3}
For every relation $w$ as in (\ref{poum}), we have
\begin{equation}w\rightsquigarrow_2 w_1w_2,\; \textnormal{ where }\;
w_1=\hat{e}_{13}^{u_1}\hat{e}_{34}^{v_1}\ldots \hat{e}_{13}^{u_q}\hat{e}_{34}^{v_q},\;
\textnormal{ and }\; w_2=\hat{e}_{12}^{u'_1}\hat{e}_{24}^{v'_1}\ldots
\hat{e}_{12}^{u'_{q'}}\hat{e}_{24}^{v'_{q'}},\label{13341224}\end{equation}
with $q+q'\le 27$.
\end{clai}
\begin{proof} In $w$, we can, using Claim \ref{clai23}, shuffle all subwords of the form $\hat{e}_{23}^{x_k}$ to the left. Some subwords of the form $\hat{e}_{13}^{y}$ or $\hat{e}_{24}^{y}$ appear: at most $27^2=729$ (although we can do much better). We can now aggregate all terms $\hat{e}_{23}^{x_k}$ to a single term $\hat{e}_{23}^x$, with quadratic cost by Claim \ref{clai2}(\ref{2vii}); since $w$ is null-homotopic, necessarily $x=0$, so we got rid of all elements $\hat{e}_{23}^{x_k}$. 

Next, using Claim \ref{clai2}, we can shuffle all subwords ($\hat{e}_{12}^{x_k}$ or $\hat{e}_{24}^{x_k}$) to the left, since they commute with quadratic cost with the subwords of the form ($\hat{e}_{13}^{x_k}$ or $\hat{e}_{34}^{x_k}$). We can aggregate when necessary consecutive subwords of the form $\hat{e}_{ij}^{x_k}$ using Claim \ref{clai2}(\ref{2vii}). Since there are at most 27 subwords of the form $\hat{e}_{12}^{x_k}$ or $\hat{e}_{34}^{x_k}$ (or 12 instead of 27 with some little effort), the claim is proved.
\end{proof}

We are led to reduce words as in (\ref{13341224}). We first need the following general formula. For commutators, we use the convention
$$[x,y]=x^{-1}y^{-1}xy.$$
In any group and for any elements $x_1,\dots y_k$, we have, denoting 
$x_{ij}=x_i x_{i+1}\dots x_j$, (if $i\le j$) and defining similarly $y_{ij}$
\begin{equation}\label{crocro}x_1y_1\dots x_ky_k\end{equation} $$=x_{1k}y_{1k}[y_{1k},x_{2k}][x_{2k},y_{2k}][y_{2k},x_{3k}]\dots [x_{(k-1)k},y_{(k-1)k}][y_{(k-1)k},x_{kk}][x_{kk},y_{kk}]$$

Thus, given $w=w_1w_2$ as in (\ref{13341224}), applying (\ref{crocro}) to each of $w_1$ and $w_2$ and using Claim \ref{clai2}(\ref{2vii}) several times, it can be reduced, with quadratic cost, to a word of the form

\[\left(\hat{e}_{13}^x\hat{e}_{34}^z\prod_{i=1}^{2(q-1)}[\hat{e}_{13}^{x_i},\hat{e}_{34}^{z_i}]^{(-1)^i}\right)\left(\hat{e}_{12}^y\hat{e}_{24}^t\prod_{i=1}^{2(q'-1)}[\hat{e}_{12}^{y_i},\hat{e}_{24}^{t_i}]^{(-1)^i}\right)\]
(observing that the number of brackets in (\ref{crocro}) is $2(k-1)$).

By projection, we have $x=y=z=t=0$. Thus we have obtained the word
\begin{equation}\prod_{i=1}^{r}[\hat{e}_{13}^{x_i},\hat{e}_{34}^{z_i}]^{(-1)^i}\prod_{i=1}^{r'}[\hat{e}_{12}^{y_i},\hat{e}_{24}^{t_i}]^{(-1)^i}.\label{1234c}\end{equation}
with $r+r'\le 52$. We have proved

\begin{clai}\label{c1234}
Every relation $w$ as in (\ref{poum}), can be reduced with quadratic cost to a relation as in (\ref{1234c}).
\end{clai}

We now recall the following formula due to Hall, valid in
any group.
\begin{equation}\label{Hallformula}
 [a^b,[b,c]]\cdot [b^c,[c,a]]\cdot [c^a,[a,b]]=1,
\end{equation}
where $a^b=b^{-1}ab.$ We also recall the simpler formula
\begin{equation}\label{proco}
 [ab,c]=[a,c]^b\,[b,c].
\end{equation}

\begin{clai}\label{hal}We have 
\begin{itemize}
\item $[\hat{e}_{13}^{x},\hat{e}_{34}^{y}] \rightsquigarrow_2 [\hat{e}_{12}^{1},\hat{e}_{24}^{xy}]$
\item $[\hat{e}_{12}^{x},\hat{e}_{24}^{y}] \rightsquigarrow_2 [\hat{e}_{13}^{1},\hat{e}_{34}^{xy}]$
\end{itemize}
\end{clai}
\begin{proof}
Both verifications are similar, so we only prove the first one.
Define $(a,b,c)=( \hat{e}_{12}^1, \hat{e}_{23}^x, \hat{e}_{34}^y)$. 
First observe that using Claim \ref{clai23}, $\hat{e}_{13}^{x}\rightsquigarrow_2 [a,b]$, so 
$$[\hat{e}_{13}^{x},\hat{e}_{34}^{y}]\rightsquigarrow_2 [[a,b],c]=[c,[a,b]]^{-1}.$$

Since $[c,a]\rightsquigarrow_2 1$, we have $[b^c,[c,a]]\rightsquigarrow_2 1$ and 
$[c,[a,b]]\rightsquigarrow_2 [c^a,[a,b]]$; thus applying Hall's formula (\ref{Hallformula}) to $(a,b,c)$ we get
$$[c,[a,b]]^{-1}\rightsquigarrow_2 [a^b,[b,c]].$$

On the other hand, using (\ref{proco}), $$[a^b,[b,c]]=[[b,a^{-1}]a,[b,c]]=[[b,a^{-1}],[b,c]]^{a}\cdot[a,[b,c]].$$
Thus using Claim \ref{clai23} and Claim \ref{clai2}(\ref{2iv}), we get
$$[a^b,[b,c]]\rightsquigarrow_2 [ \hat{e}_{13}^x, \hat{e}_{24}^{xy}]^{a}\cdot[ \hat{e}_{12}^1, \hat{e}_{24}^{xy}]\rightsquigarrow_2 [ \hat{e}_{12}^1, \hat{e}_{24}^{xy}].$$
So we proved $[\hat{e}_{13}^{x},\hat{e}_{34}^{y}] \rightsquigarrow_2 [\hat{e}_{12}^{1},\hat{e}_{24}^{xy}]$. 
\end{proof}

Define $\hat{e}_{14}^r=[ \hat{e}_{12}, \hat{e}_{24}^r]$. 

\begin{clai}\label{1424}
We have $[ \hat{e}_{14}^r, \hat{e}_{24}^x]\rightsquigarrow_2 1$.
\end{clai}
\begin{proof}
We have $[ \hat{e}_{14}^r, \hat{e}_{24}^x]\rightsquigarrow_2 [[ \hat{e}_{13}^1, \hat{e}_{34}^r], \hat{e}_{24}^x]$. Since $[ \hat{e}_{13}^1, \hat{e}_{24}^x]\rightsquigarrow_2 1$ and $[ \hat{e}_{34}^r, \hat{e}_{24}^x]\rightsquigarrow_2 1$, it follows that $[[ \hat{e}_{13}^1, \hat{e}_{34}^r], \hat{e}_{24}^x]\rightsquigarrow_2 1$.
\end{proof}

In the following claim, we use the identities, true in an arbitrary group: $[a,b]=[b,a^{-1}]^a$ and $[a,bc]=[a,c][a,b]^c$.

\begin{clai}\label{e1414}
We have $( \hat{e}_{14}^r)^{-1}\rightsquigarrow_2  \hat{e}_{14}^{-r}$ and $ \hat{e}_{14}^r \hat{e}_{14}^s\rightsquigarrow_2  \hat{e}_{14}^{r+s}$.
\end{clai}
\begin{proof}
$$( \hat{e}_{14}^r)^{-1}=[ \hat{e}_{12}^1, \hat{e}_{24}^r]^{-1}=[ \hat{e}_{24}^r, \hat{e}_{12}^1]=[ \hat{e}_{12}^1,( \hat{e}_{24}^r)^{-1}]^{ \hat{e}_{24}^r}$$
Therefore, by Claims \ref{clai2}(\ref{2vii}) and \ref{1424}, 
$$( \hat{e}_{14}^r)^{-1}\rightsquigarrow_2 [ \hat{e}_{12}^1, \hat{e}_{24}^{-r}]^{ \hat{e}_{24}^r}\rightsquigarrow_2 [ \hat{e}_{12}^1, \hat{e}_{24}^{-r}]= \hat{e}_{14}^{-r}.$$

For the addition, using Claim \ref{clai2}(\ref{2vii}) and Claim \ref{1424}
\[ \hat{e}_{14}^{r+s}=[ \hat{e}_{12}^1, \hat{e}_{24}^{r+s}]\rightsquigarrow_2 [ \hat{e}_{12}^1, \hat{e}_{24}^r \hat{e}_{24}^{s}]=[ \hat{e}_{12}^1, \hat{e}_{24}^{s}][ \hat{e}_{12}^1, \hat{e}_{24}^r]^{ \hat{e}_{24}^{s}}= \hat{e}_{14}^s( \hat{e}_{14}^r)^{ \hat{e}_{14}^s}\rightsquigarrow_2  \hat{e}_{14}^s \hat{e}_{14}^r.\qedhere\]
\end{proof}

\begin{proof}[Conclusion of the proof of Proposition \ref{quadrapei}]
By Claim \ref{c1234}, we start from a word as in (\ref{1234c}). By Claim \ref{hal}, it can be reduced with quadratic cost to a word of the form $\prod_{i=1}^{52}( \hat{e}_{14}^{r_i})^{(-1)^i}$. The inverse reduction in Claim \ref{e1414} reduces this word to $\prod_{i=1}^{52} \hat{e}_{14}^{(-1)^ir_i}$. The second one reduces it to $ \hat{e}_{14}^s$, with $s=\sum(-1)^ir_i$. Since this is a null-homotopic word, $s=0$ and we are done. 
\end{proof}

\section{Asymptotic cones and central extensions}

\subsection{Topology on the fundamental group}

In order to determine the fundamental group of some asymptotic cones, it will be useful to equip it with a group topology, and actually better, with a bi-invariant metric. We will see two possible choices for such a metric, both being potentially interesting as they provide more refined quasi-isometry invariants than the fundamental group alone.  We shall use them in order to state and prove Theorem \ref{pico} (which holds for both choices of metric). 

Let $X$ be a topological space with base-point $x_0$ with a basis of neighbourhoods $\mathcal{V}$.
A naive way to define a topology on $\pi_1(X,x_0)$ is as follows.
For every $V\in \mathcal{V}$, define $K_V$ as the set of elements representable as a loop with image in $V$; this is a subgroup of $\pi_1(X,x_0)$. However, there is, in general, no group topology on $\pi_1(X,x_0)$ such that $(K_V)_{V\in\mathcal{V}}$ is a basis of neighborhoods of~1: indeed, it is not necessarily true (e.g.\ if $X=\mathbf{R}^2-\mathbf{Q}^2$) that if $g$ is fixed and $g_i\to 1$, then $gg_ig^{-1}\to 1$ for this topology.

A natural solution is simply to replace $K_V$ by its normal closure. In other words, define $L_V$ as the set of elements in $\pi_1(X,x_0)$ that can be represented by a finite product $\prod_{i=1}^kc_i\gamma_ic_{i}^{-1}$, where $c_i,\gamma_i$ are loops based at $x_0$ and $\gamma_i$ has image in $V$. Clearly $L_V$ is a normal subgroup in $\pi_1(X,x_0)$, and  $L_V\cap L_W\supset L_{V\cap W}$ for all $V,W\in\mathcal{V}$.
It follows that the cosets of $L_V$, for $V\in\mathcal{V}$, form a basis of open (actually clopen) sets for a topology on $\pi_1(X,x_0)$, which is a group topology.

Equivalently, $\pi_1(X,x_0)$ is endowed with the topology induced by the homomorphic mapping into $\prod_{V\in\mathcal{V}}\pi_1(X,x_0)/L_V$, each $\pi_1(X,x_0)/L_V$ being discrete and the product being endowed with the product topology. Then $(X,x_0)\mapsto\pi_1(X,x_0)$ is a functor from the category of pointed topological spaces to the category of topological group. In particular, any two homeomorphic pointed topological spaces have their fundamental group isomorphic as topological groups.

If the topology of $X$ is defined by a metric $d$, this topology is pseudo-metrizable, where the pseudo-distance of a loop $\gamma$ to the identity is defined as $\inf\{\eps>0:\gamma\in L_{B(\eps)}\}$, where $B(\eps)$ is the closed $\eps$-ball around $x_0$. This pseudo-distance is bi-invariant and satisfies the ultrametric inequality. Then $(X,x_0)\mapsto\pi_1(X,x_0)$ is a functor from the category of pointed metric spaces with pointed isometric (resp. Lipschitz) maps, to the category of pseudo-metric groups. In particular, any two isometric (resp. bilipschitz) pointed metric spaces have their fundamental groups isometrically (resp. bilipschitz) isomorphic as pseudo-metric groups.

\subsection{Central subgroups and liftings}

Let $G$ be a locally compact compactly generated group and $Z$ a closed central subgroup. Fix a nonprincipal ultrafilter $\omega$ on $\mathbf{N}$ once and for all. Assume that $\Cone_\omega(Z)$ is totally disconnected, where $Z$ is always endowed with the word metric from $G$.

If $X$ is a metric space with base-point $x_0$, denote by $\mathcal{P}(X)$ the set of paths in $X$ based at $x_0$, i.e. of continuous bounded maps from $\mathbf{R}_+$ to $X$ mapping $0$ to $x_0$ (in the sequel since the considered metric spaces will be homogeneous, the choice of $x_0$ won't matter and therefore will be kept implicit).  This is a metric space with the sup distance.

There is an obvious 1-Lipschitz map $\psi:\mathcal{P}(\Cone_\omega(G))\to \mathcal{P}(\Cone_\omega(G/Z))$. As $Z$ is abelian, $\Cone_\omega(Z)$ is a topological abelian group in the natural way; moreover $Z$ being central, the action of $Z$ on $G$ by (left) multiplication induces an action of $\Cone_\omega(Z)$ on $\Cone_\omega(G)$ such that $\Cone_\omega(G/Z)$ identifies with the set of $\Cone_\omega(Z)$-orbits under this action\footnote{It turns out that this identification holds as metric spaces: the distance on $\Cone_\omega(G/Z)$ coincides with the distance between $\Cone_\omega(Z)$-orbits in $\Cone_\omega(G)$.}. Moreover, and once again because $Z$ is central, for every $x\in \Cone_\omega(G)$ and $z,z'\in \Cone_\omega(Z)$, we have
$$d(z,z')=d(zx,z'x).$$ 
In particular the action is free.

The fundamental observation is the following proposition.

\begin{prop}\label{p13}
If $\Cone_\omega(Z)$ is ultrametric, then the above map $\psi$ is a $(1/3,1)$-bilipschitz homeomorphism.
\end{prop}

Thus we have to lift paths from $\Cone_\omega(G/Z)$ to $\Cone_\omega(G)$. The easier part is uniqueness, i.e. injectivity of $u$.

Let $X\subset\Cone_\omega(G)^2$ be the graph of the equivalence relation of the action of $\Cone_\omega(Z)$. Then there is a map $s:X\to\Cone_\omega(Z)$ mapping $(x,y)$ to the unique $z$ such that $zx=y$. This map is Lipschitz: indeed if $s(x,y)=z$ and $s(x',y')=z'$, then
$$d(z,z')=d(zx,z'x)\le d(zx,z'x')+d(z'x',z'x)=d(y,y')+d(x,x').$$
We can now prove

\begin{lem}
The map $\psi$ is injective.
\end{lem}
\begin{proof}
Assume that $\psi(u)=\psi(v)$. This means that $(u(t),v(t))\in X$ for all $t$.
So there is a well-defined continuous map $\sigma:t\mapsto s(u(t),v(t))$ with values in $\Cone_\omega(Z)$, with $\sigma(0)=1$. As the latter is assumed to be totally disconnected, the map $\sigma$ has to be constant, hence equal to 1, i.e. $u=v$.
\end{proof}

\begin{lem}
The map $\psi$ has dense image.\label{psidense}
\end{lem}

To prove this we first need the following lemma.

\begin{lem}
Let $G$ be a group endowed with a word metric with respect to some generating subset $S$. Let $N$ be a closed normal subgroup, and endow $G/N$ with the word metric with respect to the image of $S$. Fix $\eps>0$. Consider $x\in\Cone_\omega(G)$ and $y,y'\in\Cone_\omega(G/N)$ satisfying $d(y,y')\le\eps$ and $p(x)=y$, where $p$ is the natural projection. Then there exists $x'\in\Cone_\omega(G)$ with $d(x,x')\le\eps$ and $p(x')=y'$.\label{liftdiscret}
\end{lem}
\begin{proof}
By homogeneity, we can suppose that $x=1$. Write $y'$ as a sequence $(y_n)$ with $\lim_\omega d(y_n,1)\le 1$, and lift $y_n$ to an element $x_n$ of $G$ with the same word length. Then $(x_n)$ defines an element $x'$ of $\Cone_\omega(G)$ with the required properties.
\end{proof}

\begin{proof}[Proof of Lemma \ref{psidense}]
Let $v$ be an element of $\mathcal{P}(\Cone_\omega(G/Z))$ and fix $\eps>0$. There exists a sequence $0=t_0<t_1<\dots$ tending to infinity such that every segment $[t_i,t_{i+1}]$ is mapped by $v$ to a set of diameter at most $\eps$. By applying inductively Lemma \ref{liftdiscret}, there exists a sequence $(x_i)$ in $\Cone_\omega(G)$ such that $x_0=1$, $p(x_i)=v(t_i)$ and $d(x_i,x_{i+1})\le\eps$ for all $i$. As $\Cone_\omega(G)$ is geodesic, we can find a continuous function $u:\mathbf{R}_+\to\Cone_\omega(G)$ such that $u(t_i)=x_i$ for all $i$ and $u$ is geodesic on every segment $[t_i,t_{i+1}]$ (in the sense that for some constant $c_i\ge 0$ and all $t,t'$ in this segment, $d(u(t),u(t'))=c_i|t-t'|$). Then $d(p\circ u,v)\le 2\eps$, and observe that $\psi(u)=p\circ u$ by definition of $\psi$.
\end{proof}

Let $\mathcal{P}_g$ denote the set of elements $u$ of $\mathcal{P}(\Cone_\omega(G))$
such that there exists $0=t_0<t_1<\dots$ tending to infinity such that $u$ is geodesic in restriction to each $[t_i,t_{i+1}]$ and such that the projection $$p:\Cone_\omega(G)\to\Cone_\omega(G/Z)$$ is isometric in restriction to $u([t_i,t_{i+1}])$. We summarize this as: the sequence $0=t_0<t_1<\dots$ is $u$-good; if moreover $d(u(t_i),u(t_{i+1}))\le\eps$ for all $i$, we call it $(u,\eps)$-good; note that given $\eps>0$, any $u$-good sequence can be refined to a $(u,\eps)$-good sequence. The proof of Lemma \ref{psidense} actually shows that $\psi(\mathcal{P}_g)$ is dense in $\mathcal{P}(\Cone_\omega(G/Z))$.

\begin{lem}
Suppose that $\Cone_\omega(Z)$ is ultrametric. Then the map $\psi^{-1}$ is 3-Lipschitz on $\psi(\mathcal{P}_g)$.\label{troislip}
\end{lem}
\begin{proof}
Let $u,v$ belong to $\mathcal{P}_g$ with $d(p\circ u,p\circ v)=\sigma$. Fix $\eps>0$. Consider a sequence $0=t_0<t_1<\dots$ which is both $(u,\eps)$ and $(v,\eps)$-good. By Lemma \ref{liftdiscret}, there exists $w_i$ such that $p(w_i)=p(u(t_i))$ and $d(w_i,v(t_i))\le\sigma$ (we choose $w_0=1$). Set $z_i=s(u(t_i),w_i)$, i.e. $w_i=z_iu(t_i)$. Then $$d(z_i,z_{i+1})=d(z_iu(t_i),z_{i+1}u(t_i))\le d(z_iu(t_i),v(t_i))+d(v(t_i),v(t_{i+1}))+$$ $$d(v(t_{i+1}),z_{i+1}u(t_{i+1}))+d(z_{i+1}u(t_{i+1}),z_{i+1}u(t_i))$$
$$\le\sigma+\eps+\sigma+\eps.$$
As $\Cone_\omega(Z)$ is ultrametric, we obtain that $d(1,z_i)\le 2\sigma+2\eps$ for all $i$. So $d(w_i,u(t_i))\le 2\sigma+2\eps$ for all $i$, and so $d(u(t_i),v(t_i))\le 3\sigma+2\eps$ for all $i$. Accordingly, $d(u(t),v(t))\le 3\sigma+3\eps$ for all $t$. As $\eps$ is arbitrary, we obtain  $d(u(t),v(t))\le 3\sigma$.
\end{proof}

\begin{proof}[Proof of Proposition \ref{p13}]It follows from Lemma \ref{troislip} that $\psi^{-1}$ extends to a 3-Lipschitz map $\varphi$ defined on the closure of $\psi(\mathcal{P}_g)$, which is by Lemma \ref{psidense} all of $\mathcal{P}(\Cone_\omega(G/Z))$; moreover by density and continuity, $\psi\circ\varphi$ is the identity on $\mathcal{P}(\Cone_\omega(G/Z))$. So $\psi$ is surjective and is duly $(1/3,1)$-bilipschitz.
\end{proof}

If $X$ is a metric space with base-point $x_0$, denote by $\mathcal{L}(X)$ the set of continuous loops based on $x_0$. It can be naturally viewed as a closed metric subspace of $\mathcal{P}(X)$ by extending all functions by the constant function equal to the base-point on $[1,\infty[\,$. In particular, all the above can be applied.

Take again $G,Z\dots$ as above and keep assuming that $\Cone_\omega(Z)$ is ultrametric. If $u\in \mathcal{L}(\Cone_\omega(G/Z))$, define $\mu(u)=\varphi(u)(1)$. This defines a $3$-Lipschitz map $$\mu:\mathcal{L}(\Cone_\omega(G/Z))\to\Cone_\omega(Z).$$
In particular, being continuous and mapping to a totally disconnected space, it factors through a map $$\tilde{\mu}:\pi_1(\Cone_\omega(G/Z))\to\Cone_\omega(Z).$$
Clearly, $\mu=1$ in restriction to $\psi(\mathcal{L}(\Cone_\omega(G)))$. Therefore the following composition is trivial
$$\pi_1(\Cone_\omega(G))\to\pi_1(\Cone_\omega(G/Z))\stackrel{\tilde{\mu}}\to\Cone_\omega(Z).$$

The map $\tilde{\mu}$ is surjective: this is a trivial consequence of the path-connectedness of $\Cone_\omega(G)$.

The lifting map $\varphi=\psi^{-1}$ allows to lift homotopies and as a direct consequence we get the injectivity of the map $\pi_1(\Cone_\omega(G))\to\pi_1(\Cone_\omega(G/Z))$.

Finally, if $u$ is in the kernel of $\tilde{\mu}$, then this means that $\varphi(u)$ is a loop of which $u$ is the image.
So we get an exact sequence of groups
\begin{equation}1\to\pi_1(\Cone_\omega(G))\to\pi_1(\Cone_\omega(G/Z))\stackrel{\tilde{\mu}}\to\Cone_\omega(Z)\to 1.\label{supse}\end{equation}

It is actually, in a reasonable sense, an exact sequence in the context of metric groups with Lipschitz maps. 

\begin{defn}
Given three metric groups (groups endowed with left-invariant pseudometrics), we call an exact sequence $$1\to N\stackrel{\iota}\longrightarrow G\stackrel{p}\longrightarrow Q\to 1$$ {\em Lipschitz-exact}\footnote{We do not require $\iota$ to be a bilipschitz embedding, so this could be called ``right Lipschitz-exact exact sequence"; however we shall not use this stronger notion of being Lipschitz-exact.}
 if $\iota$ is Lipschitz, and there exist constants $C,C'>0$ such that \begin{equation}Cd(g,\text{Ker}(p))\le d(1,p(g))\le C'd(g,\text{Ker}(p))\label{cccc}\end{equation} for all $g\in G$.
\end{defn}
Inequality (\ref{cccc}) says that the distance between elements in $Q$ is bi-Lipschitz equivalent to the distance between corresponding $N$-cosets in $G$. 
In particular the right-hand inequality in (\ref{cccc}) means that $p$ is $C'$-Lipschitz. Observe that if $N$  is trivial, (\ref{cccc}) just means that $p:G\to Q$ is a bilipschitz isomorphism. 

In our case, the exact sequence (\ref{supse}) is Lipschitz-exact with constants 1 and 3. The right-hand inequality follows from the combined facts that $Z$ is commutative (hence conjugations disappear in the image) and ultrametric (so that a large product of small loops is still small). 

To check the non-trivial left-hand case, take $u\in\pi_1(\Cone_\omega(G/Z))$. Fixing a representing element in $\mathcal{L}(\Cone_\omega(G/Z))$, lift it (through $\varphi$), extend it to a closed loop via a geodesic of length $d(1,\tilde{\mu}(u))$, and take the image by $p$. We get an element of $\text{Ker}(\tilde{\mu})$ at distance $\le d(1,\tilde{\mu}(u))$ of $u$.

Finally we get

\begin{thm}\label{pico}
Let $G$ be a locally compact, compactly generated group and $Z$ a closed, central subgroup. Endow $G$ with a word length with respect to a compact generating subset and let $Z$ be endowed with the restriction of this word length. Given a nonprincipal ultrafilter $\omega$, assume that $\Cone_\omega(Z)$ is ultrametric. Then the sequence
$$1\to\pi_1(\Cone_\omega(G))\to\pi_1(\Cone_\omega(G/Z))\stackrel{\tilde{\mu}}\longrightarrow\Cone_\omega(Z)\to 1$$
is a Lipschitz-exact sequence of metric groups.
\end{thm}

\begin{cor}\label{conab}
If $\Cone_\omega(G)$ is simply connected, then 
$$\pi_1(\Cone_\omega(G/Z))\stackrel{\tilde{\mu}}\longrightarrow\Cone_\omega(Z)$$
is a bilipschitz isomorphism of metric groups.
\end{cor}

\subsection{Ultrametric on Abels' group}\label{ula}

Let $\mathbf{K}$ be a nondiscrete locally compact field. On $\SL_d(\mathbf{K})$, define a left-invariant pseudometric $d(g,h)=\ell(g^{-1}h)$, where $\ell$ is the length defined as follows
\begin{itemize}
\item If $\mathbf{K}$ is ultrametric, $\ell(A)=\sup_{i,j}\log|A_{ij}|$;
\item if $\mathbf{K}$ is Archimedean, $\ell(1)=0$ and $\ell(A)=\sup_{i,j}\log|A_{ij}|+C$, where $C$ is a large enough constant ($C\ge\log d$ works).
\end{itemize}
This length is equivalent to the word length with respect to a compact generating subset. Moreover, the embedding $A_4(\mathbf{K})\subset \SL_4(\mathbf{K})$ is quasi-isometric, therefore we can endow $A_4(\mathbf{K})$ (or any cocompact lattice therein) with the restriction of this distance, which is equivalent to the word distance. 

We immediately see that this distance is ultrametric in restriction to $Z(\mathbf{K})\subset A_4(\mathbf{K})$ if $\mathbf{K}$ is ultrametric, and quasi-ultrametric in the case of $\mathbf{K}$ Archimedean, namely satisfies $d(x,z)\le\max(d(x,y),d(y,z))+\log(2)$. Thus, in all cases, $\Cone_\omega(Z(\mathbf{K}))$ is ultrametric.

Thus, Corollary \ref{conab} can be applied along with Theorem \ref{ak}. This yields Corollary \ref{aba}.

\begin{cor}\label{comoc}
All asymptotic cones of the group $G_\mathbf{K}=A_4(\mathbf{K})/Z(\mathbf{K})$ have an abelian fundamental group with continuum cardinality. Precisely, the fundamental group of $\Cone_\omega(G_\mathbf{K})$ is isomorphic, as an abstract group, to $\mathbf{F}^{(\mathbf{R})}$ (direct sum of continuum copies of $\mathbf{F}$), where $\mathbf{F}$ is the prime field of the same characteristic as $\mathbf{K}$, namely $\mathbf{Q}$ or $\F_p$.
\end{cor}

Actually, the fundamental group is isomorphic, as a {\em topological} group, to $(\mathbf{F}^{(\epsilon_1)})^{\epsilon_0}$, a countable product of the direct sum of continuum copies of $\mathbf{F}$. This is established in Corollary \ref{carpc} when $\mathbf{K}$ has characteristic $p$, and Proposition \ref{car0c} the remaining case $\mathbf{F}=\mathbf{Q}$.

\section{Examples with lattices}\label{ewl}

Let $R$ be either $\mathbf{F}_p[t,t^{-1},(t-1)^{-1}]$, or the ring of integers of a totally real number field of degree 3.

\begin{thm}\label{a4r}
The group $A_4(R)$ is finitely presented and has a quadratic Dehn function.
\end{thm}

The proof consists of embedding $R$ as a cocompact lattice in a certain larger group. There is a natural cocompact embedding $R\subset\mathbf{K}$, where $\mathbf{K}$ is the locally compact ring $\mathbf{R}^3$ or $\mathbf{F}_p(\!(t)\!)^3$, given by 

\[f(z) =	\left\{ \begin{array}{rcl} P(t)\mapsto (P(t),P(1-t),P(t^{-1})) & \mbox{if }  R=\mathbf{F}_p[t,t^{-1},(t-1)^{-1}]; \\ 
x\mapsto (\sigma_1(x),\sigma_2(x),\sigma_3(x)) & \mbox{in the real case,} 
\end{array}\right.\]
where in the second case $\sigma_1,\sigma_2,\sigma_3$ are the three distinct real field embeddings of the fraction field of $R$ into $\mathbf{R}$.

 Note that $A_4(R)$ is not cocompact in $A_4(\mathbf{K})$, because $R^\times$ is not cocompact in $\mathbf{K}^\times$. However $R^\times$ is cocompact in $\mathbf{K}^\times_1$, the closed subgroup of elements in $\mathbf{K}^\times$ for which the multiplication preserves the Haar measure in $\mathbf{K}$. Let $A_4(\mathbf{K})_1$ be the set of elements of $A_4(\mathbf{K})$, both of whose diagonal entries are in $\mathbf{K}^\times_1$, so $A_4(R)$ is cocompact in $A_4(\mathbf{K})_1$. Theorem \ref{a4r} follows from 

\begin{thm}\label{a41}
$A_4(\mathbf{K})_1$ has a quadratic Dehn function.
\end{thm}

\noindent {\it On the proof.} The proof is strictly analogous to that of $A_4$ of a nondiscrete locally compact field, so we do not repeat all the technical details. The only difference lies in the adaptation of the proof of Claims \ref{clai23} and \ref{clai2}. It could be proved by using the natural generalization of Lemma \ref{quadratic_subgroup_claim}. However this would be more difficult as there is no ``contracting element" in the subgroups involved, because of the restriction on the determinant of the diagonal elements. So we use a variant of Lemma \ref{quadratic_subgroup_claim} rather than its generalization. 

Recall that $\mathbf{K}$ is now a product of three nondiscrete locally compact fields, say $\mathbf{K}=\mathbf{K}^1\times\mathbf{K}^2\times\mathbf{K}^3$. Recall from \S\ref{ppq} that $J=\{(i,j), 1\leq i<j\leq 4, (i,j)\neq (1,4)\}$. 
Define $J'=J\times\{1,2,3\}$, and let, for $(i,j,m)\in J'$, $U_{ijm}$ be the set of elements in $U_{ij}(\mathbf{K})$ whose $(i,j)$-entry is in $\mathbf{K}^m$. 
The adapted version of Claims \ref{clai23} and \ref{clai2} is the following:
\begin{clai}
For all $(i,j,m)\in J'$, we have
\begin{itemize}
\item $\hat{e}_{ijm}^x\hat{e}_{ijm}^y \rightsquigarrow_2\hat{e}_{ijm}^{x+y}$; $(\hat{e}_{ijm}^x)^{-1}\rightsquigarrow_2\hat{e}_{ijm}^{-x}$;
\item if $(k,\ell,n)\in J'$ with $(j,m)\neq (k,n)$, $[\hat{e}_{ijm}^x,\hat{e}_{k\ell n}^y]\rightsquigarrow_21$;
\item if $(j,k,m)\in J'$ and $(i,k)\neq(1,4)$, $[\hat{e}_{ijm}^x,\hat{e}_{jkm}^y]\rightsquigarrow_2\hat{e}_{ikm}^{xy}$.
\end{itemize}
\end{clai}

It is proved by showing that each of these relations hold in a smaller subgroup with a quadratic Dehn function, using the following substitute for Lemma \ref{quadratic_subgroup_claim}:

\begin{lem}
For all $(i,j,m),(k,\ell,n)\in J'$, such that $(j,m)\neq (k,n)$, $DU_{ijm}U_{k\ell n}$ has a quadratic Dehn function; for all $(i,j,m),(j,k,m)\in J'$ with $(i,k)\neq (1,4)$, $DU_{ijm}U_{jkm}U_{ikm}$ has a quadratic Dehn function.
\end{lem} 
\begin{proof}[Proof (sketched)]
The point is, for each of these groups, to find in $D$ a contracting element. Most cases were already considered in the proof of Lemma \ref{quadratic_subgroup_claim}. The only remaining ones are $(1,j,m)$ and $(j,4,n)$ with $m\neq n$ and $j\in\{2,3\}$. To do this, it is enough to find an element in $\mathbf{K}^\times_1$ contracting $\mathbf{K}^m$ and dilating $\mathbf{K}^n$. This is well-known and was already done in \cite{CT}.
\end{proof}

\medskip

\begin{proof}[Proof of Corollary \ref{cpoly}]
The proof of Corollary \ref{comoc}, without changes, proves that $\pi_1(\Cone_\omega(A_4(\mathbf{R}^3)_1/Z(\mathbf{R}^3)))$ is isomorphic, as an abstract group, to $\mathbf{Q}^{(\mathbf{R})}$. Since $\Lambda$ is a lattice in $A_4(\mathbf{R}^3)_1/Z(\mathbf{R}^3)$, the same result holds for $\pi_1(\Cone_\omega(\Lambda))$.

By \cite[Corollary 3.$\textnormal{F}'_5$]{Gro}, every connected solvable Lie group has an at most exponential Dehn function, so every polycyclic group has an at most exponential Dehn function. Since $Z(\mathbf{R}^3)$ is central and exponentially distorted, for every path $\gamma$ of linear size in $A_4(\mathbf{R}^3)_1$ joining the identity to an element of exponential size in $Z(\mathbf{R}^3)$, the image of $\gamma$ in $A_4(\mathbf{R}^3)_1/Z(\mathbf{R}^3)$ has an (at least) exponential area. Thus $A_4(\mathbf{R}^3)_1/Z(\mathbf{R}^3)$, and therefore $\Lambda$ as well, has an at least exponential Dehn function. Finally, we conclude that $\Lambda$ has an exactly exponential Dehn function. 
\end{proof}

Corollary \ref{a4p} follows from Theorem \ref{a41}, because $A_4(\F_p[t,t^{-1},(t-1)^{-1}])$ is a cocompact lattice in $A_4(\mathbf{F}_p(\!(t)\!)^3)_1$.

\begin{rem}\label{a4notfg}
The group $A_4(\F_p[t])$ is not finitely generated. Indeed, consider its subgroup $H$ of matrices $(a_{ij})$ satisfying $a_{22}=1$. Since the group of invertible elements in $\F_p[t]$ is reduced to $\F_p^\times$, $H$ has finite index (namely, $p-1$) in $A_4(\F_p[t])$. There is a surjective homomorphism $H\to\F_p[t]$, mapping a matrix as above to $a_{12}$. Since $\F_p[t]$ is not finitely generated as a group, it follows that $H$ is not finitely generated, nor is its overgroup of finite index $A_4(\F_p[t])$.

For similar but more elaborate reasons, $A_4(\F_p[t,t^{-1}])$ is not finitely presented (its finite generation is an easy exercise). Indeed, define $L$ as those matrices $(a_{ij})$ for which $a_{22}$ is some power of $t$. Since the group of invertible elements in $\F_p[t,t^{-1}]$ consists of $ut^k$ for $u\in\F_p^\times$ and $k\in\mathbf{Z}$, $L$ has index $p-1$ in $A_4(\F_p[t,t^{-1}])$. There is a homomorphism $\phi$ from $L$ to $\textnormal{GL}_2(\F_p[t,t^{-1}])$, mapping a matrix to its $2\times 2$ northwest block, so that $$\phi(L)=\left\{\begin{pmatrix}1 & Q\\ 0 & t^k\end{pmatrix}:\;Q\in\F_p[t,t^{-1}],k\in\mathbf{Z}\right\}.$$
The latter group is isomorphic to the lamplighter group $\F_p\wr\mathbf{Z}$ and is not a quotient of a finitely presented solvable group \cite{Bau,BS}. So $L$ is not finitely presented, nor is its overgroup of finite index $A_4(\F_p[t,t^{-1}])$.

Similarly, using that the Baumslag-Solitar group $\mathbf{Z}[1/p]\rtimes\mathbf{Z}$ is a retract of a subgroup of finite index in Abels' original group $A_4(\mathbf{Z}[1/p])$, and using that $\mathbf{Z}[1/p]\rtimes\mathbf{Z}$ has an exponential Dehn function (see for instance \cite{GH}), it follows that $A_4(\mathbf{Z}[1/p])$ has an at least exponential Dehn function. Actually, it is not hard to check that the Dehn function of $A_4(\mathbf{Z}[1/p])$ is exponential.

Let us also mention that all the results given here about $A_4$ extend with mechanical changes to higher dimensional Abels groups $A_d$ for $d\ge 4$; on the other hand $A_3(\mathbf{K})$ is not compactly presented if $\mathbf{K}$ is a non-Archimedean local field, again using \cite{BS}.
\end{rem}

\section{Cones of subgroups of  $\F_p^{(\mathbf{N})}$}\label{csds}

Let $\mathbf{N}$ be the set of positive integers (so $0\notin\mathbf{N}$).
Consider the group $\F_p^{(\mathbf{N})}$, with basis $(\delta_n)_{n\in\mathbf{N}}$ as a vector space over $\F_p$. We endow it with the left-invariant ultrametric distance defined by the length $|u|=\sup\{n:u_n\neq 0\}$. Observe that the ball of radius $n$ in $\F_p^{(\mathbf{N})}$ is equal to $\F_p^{\{1,\dots,n\}}$ and has cardinality $p^n$.
For every subset $J$ of $\mathbf{N}$, we endow the subgroup $\F_p^{(J)}$ with the induced metric.

The purpose of this section is to study the metric groups $\Cone_\omega \big(\F_p^{(J)}\big)$ in terms of properties of the subset $J$ and of the ultrafilter $\omega$. This analysis notably includes a classification of these asymptotic cones up to isomorphism of topological groups (see Theorem \ref{thm:cla}).

The relevance of this study comes from Corollary \ref{conab}: indeed, since $\F_p^{(J)}$ can be viewed as a bilipschitz embedded central subgroup of $A_4(\F_p[t,t^{-1},(t-1)^{-1}])$, there is a bilipschitz group isomorphism
\begin{equation}  \pi_1\left(\Cone\left(\Gamma_J\right)\right)\simeq \Cone_\omega \left(\F_p^{(J)}\right),
\label{identif}\end{equation}
where
\[\Gamma_J=A_4(\F_p[t,t^{-1},(t-1)^{-1}])/\F_p^{(J)}.\]

\subsection{Subsets of $\mathbf{N}$}
In this part, $J$ usually denotes a nonempty subset of $\mathbf{N}$ and is endowed with the induced distance from the reals.
Here we study asymptotic metric properties of $J$, which will be needed in \S\ref{fpj} in the study of the metric group $\F_p^{(J)}$.

Throughout the sequel, we use that the inclusion $J\subset\mathbf{R}$ induces an isometric 
identification of $\Cone_\omega(J)$ with a closed subset of $\mathbf{R}$ containing 0, mapping $(u_n)$ to $\lim_\omega u_n/n$.
 The motivation is the fact that the set of distances in the metric group $\Cone\big(\F_p^{(J)}\big)$ is exactly the asymptotic cone $\Cone_\omega(J)$ (see \S\ref{fpj}).

If $x\in\mathbf{N}$, define 
\[\sigma_J(x)=\inf_{y\in J}|\log(x/y)|.\] It measures the ``multiplicative distance" from $x$ to $J.$

\begin{lem}\label{critcone2}
Let $J$ be a nonempty subset of $\mathbf{N}$ and $\omega$ a nonprincipal ultrafilter. Then $\Cone_\omega(J)=\{0\}$ if and only if $\lim_{n\to\omega}\sigma_J(n)=\infty$.
\end{lem}
\begin{proof}
Suppose that $\lim_{\omega}\sigma_J=\ell<\infty$. So there is $M\in\omega$ with $\sigma_J(n)\le 2\ell$ for all $n\in M$. For all $n\in M$, let $u(n)$ be an element of $J$ with $|\log(n/u(n))|\le 2\ell$, in other words, $u(n)\in [e^{-2\ell}n,e^{2\ell}n]$. If $n\notin M$ define $u(n)=\inf(J)$. So $|u(n)|\le e^{2\ell}n$ for all $n$, so that $(u(n))$ defines an element of $\Cone_\omega(J)$; moreover its absolute value is at least $e^{-2\ell}$ so it is a nonzero element.

Conversely, assume that $\Cone_\omega(J)\neq\{0\}$; let $(u(n))$ be a nonzero element, so $u(n)\le Cn$ for all $n$ and $\lim_\omega |u(n)|/n=c>0$. 
Define $M=\{n\ge 1:|u(n)|\ge cn/2\}$, so $M\in\omega$. For $n\in M$, $u(n)\in J$ and thus $\sigma_J(n)\le\max(|\log(c/2)|,|\log C|)$, thus $\lim_\omega\sigma_J<\infty$.
\end{proof}

Consider now any quickly growing set of positive integers, i.e., $\{\alpha_m:m\ge 0\}$, where $\alpha_{m+1}/\alpha_m\to\infty$. Let $(I_n)_{n\in \mathbf{N}}$ be an infinite partition of this set into infinite sets. For every $n\in \mathbf{N},$ let $\omega_n$ be an ultrafilter supported by $I_n$. 

\begin{prop}\label{pleino2}
The map
\begin{eqnarray*}
\Phi:\{\textnormal{Subsets of $\mathbf{N}$}\} & \to & \{\textnormal{Subsets of $\mathbf{N}$}\}\\ J & \mapsto & \{n:\textnormal{Cone}_{\omega_n}(J)=\{0\}\}
\end{eqnarray*}
is surjective.
\end{prop}
\begin{proof}
Let $M$ be a subset of $\mathbf{N}$ and $J(M)=\bigcup_{n\in M}I_n$. Let us check that 
\[\{n\in\mathbf{N}:\textnormal{Cone}_{\omega_n}(J(M))\neq\{0\}\}=M,\]
so that $\Phi(J(\mathbf{N}\smallsetminus M))=M$.
 Indeed, if $n\in M$ then $I_n\subset J(M)$, so $\sigma_{J(M)}(m)=0$ for all $m\in I_n$
so $\lim_{\omega_n}\sigma_{J(M)}=0$ and $\Cone_{\omega_n}J(M)$ is nonzero by Lemma \ref{critcone2}. Conversely, if $n\notin M$, then
 $\lim_{m\in I_n,m\to\infty}\sigma_{J(M)}(m)=\infty$ (because $I$ is a fastly growing set), so $\lim_{\omega_n}\sigma_{J(M)}=\infty$, so $\Cone_{\omega_n}J(M)$ is zero, again by Lemma \ref{critcone2}. 
\end{proof}

For any subset $J\subset\mathbf{N}$, $C>1>c$, and $n\in\mathbf{N}$, let
\begin{itemize}
\item $\rho_C(n)$ be the smallest element in $J\cap [Cn,\infty\mathclose[\cup\{\infty\}$;
\item
$\lambda_c(n)$ be the largest element in $\{0\}\cup J\cap [0,cn]$.
\end{itemize}

\begin{defn}
We say that $J$ is $\omega$-{\em bounded} if $\lim_\omega\rho_C(n)/n=\infty$ for some $C>1$, and is $\omega$-{\em discrete} if $\lim_\omega\lambda_c(n)/n=0$ for some $c<1$.
\end{defn}

\begin{lem}\label{cp_bound2}
Let $J$ be a nonempty subset of $\mathbf{N}$. Then each of the following properties hold
\begin{enumerate}
\item $\Cone_\omega(J)$ is bounded if and only if $J$ is $\omega$-bounded, 
\item 0 is isolated in $\Cone_\omega(J)$ if and only if 
$J$ is $\omega$-discrete.
\end{enumerate}
\end{lem}
\begin{proof}
Suppose that $\lim_\omega\rho_C(n)/n=\infty$. Then $\Cone_\omega\big(\F_p^{(J)}\big)$ is equal to its closed $C$-ball. Conversely if $\Cone_\omega\big(\F_p^{(J)}\big)$ is equal to its open $C$-ball, then $\lim_\omega\rho_C(n)/n=\infty$. The verifications are straightforward and the other equivalence is similar.
\end{proof}

Recall that two subsets $A,B$ of $\mathbf{R}$ are at finite Hausdorff distance if there exists $c\ge 0$ such that $A\subset B+[-c,c]$ and $B\subset A+[-c,c]$.

The following lemma is motivated by the observation that if two metric spaces $X,Y$ are bilipschitz homeomorphic, then the set of logarithms of their nonzero distances are at finite Hausdorff distance (see \S\ref{fpj}). It is easy to construct continuum many closed discrete subsets of $\mathbf{R}$ at pairwise infinite Hausdorff distance. The following lemma gives a little improvement of this.

\begin{lem}\label{bilcon}
There exists a subset $K\subset\mathbf{N}$ such that the subsets
\[\log(\Cone_\omega(K)\smallsetminus\{0\})\subset\mathbf{R}\]
achieve continuum many subsets of $\mathbf{R}$ at pairwise infinite Hausdorff distance, when $\omega$ ranges over nonprincipal ultrafilters.
\end{lem}

We need the following lemma, which will be used several times in the sequel.

\begin{lem}\label{ispec2}
There exists a subset $I\subset\mathbf{N}$ such that for every $n$ and every $J\subset\{1,\dots,n\}$, there is $m\in\mathbf{N}$ satisfying $m\ge n$ and $I\cap [n^{-1}m,nm]=\{jm:j\in J\}$.
\end{lem}
\begin{proof}
Construct $I$ by an obvious induction, after enumerating finite subsets of $\mathbf{N}$.
\end{proof}

\begin{proof}[Proof of Lemma \ref{bilcon}]
Let us use the notation of the proof of Proposition \ref{pleino2}. Let us first observe that for any two distinct subsets $M,M'$ of $\mathbf{N}$, the subsets $\log(J(M))$ and $\log(J(M'))$ are at infinite Hausdorff distance. 

So to prove the result, let us construct $I$ such that for every $M\subset\mathbf{N}$ there exists an ultrafilter satisfying $\Cone_\omega(I)=\{0\}\cup J(M)$. More generally, let us show that the set $I$ of Lemma \ref{ispec2} satisfies: for every subset $L$ of $\mathbf{N}$, there exists a nonprincipal ultrafilter $\omega$ such that $\Cone_\omega(I)=\{0\}\cup L$. Indeed, write, in Lemma \ref{ispec2}, $m=m(n,J)$. Fix $L$ and define $m_n=m(n,L\cap\{1,\dots,n\})$. Then $\frac{1}{m_n}I\cap [n^{-1},n]=L\cap\{1,\dots,n\}$. Therefore for every nonprincipal ultrafilter $\omega$ containing $\{m_n:n\in\mathbf{N}\}$, we have $\Cone_\omega(I)=\{0\}\cup L$. 
\end{proof}

We will also use some further asymptotic features of subsets $J$ of $\mathbf{N}$, relevant to the study of asymptotic cones of $\F_p^{(J)}$, that are not necessarily reflected in the asymptotic cones of $J$ itself.

\begin{defn}Let $J$ be a subset of $\mathbf{N}$.
We say that $J$ is
\begin{itemize}
\item $\omega$-{\em doubling} if for all $0<c<C<\infty$ we have 
\[\lim_\omega \#([cn,Cn]\cap J)<\infty\]
\item lower $\omega$-{\em semidoubling} if for some $C_0$, this is true for all $C\le C_0$ and all $c<C$;
\item upper $\omega$-{\em semidoubling} if for some $c_0$, this is true for all $c\ge c_0$ and all $C>c$. 
\end{itemize}
\end{defn}

(Note that the terminology here corresponds to the growth of $\log(J)$ rather than $J$. The point is that we are especially motivated by the growth in $\F_p^{(J)}$.) 

\begin{lem}
For every subset $J\subset\mathbf{N}$ and nonprincipal ultrafilter $\omega$, we have the implications
\begin{itemize}
\item $\omega$-doubling $\Rightarrow$ lower $\omega$-semidoubling $\Leftarrow$ $\omega$-discrete
\item $\omega$-doubling $\Rightarrow$ upper $\omega$-semidoubling $\Leftarrow$ $\omega$-bounded.
\end{itemize}
\end{lem}
\begin{proof}
Since this lemma will not be used directly and will be obtained indirectly as a consequence of Lemma \ref{cotr} and Proposition \ref{doubca}, we leave the proof to the reader. Using these lemmas, it just corresponds to the implications for arbitrary metric spaces
\begin{itemize}
\item proper $\Rightarrow$ locally compact $\Leftarrow$ discrete;
\item proper $\Rightarrow$ $\sigma$-bounded  $\Leftarrow$ compact.\qedhere
\end{itemize}
\end{proof}

We also need the following notion

\begin{defn}
Let $J$ be a subset of $\mathbf{N}$ and $\omega$ a nonprincipal ultrafilter. Define the $\omega$-width of $J$ as
\[\wid_\omega(J)=\sup_{C>1}\lim_{n\to\omega}\#([C^{-1}n,Cn]\cap J)\;\in\{0,1,\dots,\infty\}.\]
 \end{defn}

For instance, $\mathbf{N}$ has infinite width for every nonprincipal ultrafilter; the same is true for the set of powers of 2. Thus finite width means very sparse in a certain sense.

By Lemma \ref{critcone2}, $\wid_\omega(J)=0$ if and only if $\Cone_\omega(J)=\{0\}$. The definition of the width can be essentially restated as

\begin{lem}\label{wid}
Let $J$ be a subset of $\mathbf{N}$ and $\omega$ a nonprincipal ultrafilter. Given $q\in\mathbf{N}\cup\{0\}$, we have
\begin{itemize}
\item  $\wid_\omega(J)\le q$ if and only if for every $C>1$ and $\omega$-almost all $n\in\mathbf{N}$ we have $\#([C^{-1}n,Cn]\cap J)\le q$;
\item $\wid_\omega(J)=q$ if and only if for every large enough $C$, for $\omega$-almost all $n$ we have $\#([C^{-1}n,Cn]\cap J)=q$;
\item $\wid_\omega(J)\ge q$ if and only if for some $C>1$, for $\omega$-almost all $n$ we have $\#([C^{-1}n,Cn]\cap J)\ge q$.
\end{itemize}
\end{lem}
\begin{proof}
The first assertion is a restatement of the definition, and the two others are obviously equivalent.
\end{proof}

\begin{lem}
Let $J$ be a subset of $\mathbf{N}$. Then $\wid_\omega(J)<\infty$ if and only if $J$ is $\omega$-doubling, $\omega$-bounded and $\omega$-discrete. Moreover, these conditions imply that $\Cone_\omega(J)$ has at most $\wid_\omega(J)+1$ elements.
\end{lem}
\begin{proof}
Since this lemma will not be used directly and will be obtained indirectly as a consequence of Proposition \ref{widf}, we omit the proof (the equivalence corresponding to the equivalence, for an arbitrary metric space: finite $\Leftrightarrow$ proper, discrete and bounded).
\end{proof}

The following proposition contains a list of behaviors, which is not comprehensive, but is enough to be applied to Theorem \ref{thm:cla}.

\begin{prop}\label{allk}
There exists a subset $K\subset\mathbf{N}$ satisfying: for {\em each} of the following conditions, there is a nonprincipal ultrafilter $\omega$ with the required properties:
\begin{enumerate}
\item\label{al1} $K$ has $\omega$-width $k\in\mathbf{N}\cup\{0\}$;
\item\label{al2} $K$ is $\omega$-discrete, $\omega$-doubling and not $\omega$-bounded;
\item\label{al3} $K$ is $\omega$-discrete, and not upper $\omega$-semidoubling;
\item\label{al4} $K$ is $\omega$-doubling, $\omega$-bounded, and not $\omega$-discrete;
\item\label{al5} $K$ is $\omega$-doubling, not $\omega$-bounded and not $\omega$-discrete;
\item\label{al6} $K$ is lower $\omega$-semidoubling, not $\omega$-discrete and not $\omega$-doubling;
\item\label{al7} $K$ is not lower $\omega$-semidoubling.
\end{enumerate}
\end{prop}
\begin{proof}
Let $I$ be given by Lemma \ref{ispec2}.
It follows from its definition (applied to $16^n$ rather than $n$) that for every $n\in\mathbf{N}$ and $J\subset\{1,\dots,16^n\}$, there exists $m=m(J,n)$ such that $I\cap [2^{-n}m,16^nm]=\{jm:j\in J\}$.

Fix $k\ge 0$ and define, $m_n=m(\{1,\dots,k\},n)$ for $n\ge \lceil\log_{16}(k)\rceil$. Let $\omega$ be an ultrafilter containing $\{m_n:n\ge \lceil\log_{16}(k)\rceil\}$. Thus for every $n\ge \lceil\log_{16}(k)\rceil$ we have
\[\frac{1}{m_n}I\cap [2^{-n},2^n]=\{1,\dots,k\}.\]
Then by Lemma \ref{wid}, $I$ has $\omega$-width $k$ as required in (\ref{al1}).

Now set $K_n=\{2^j:n\le j\le 3n\}$ and $\mu_n=m(K_n,n)$. Thus for every $n\ge 1$ we have \[\frac{1}{\mu_n}I\cap [1,16^n]=\{2^j:n\le j\le 3n\}.\]
In particular, dividing by $2^n$, $4^n$ and $8^n$ we get
\begin{align}
\label{eee1}\frac{1}{2^n\mu_n}I\cap [2^{-n},2^n]=&\{2^j:0\le j\le n\};\\
\label{eee2}\frac{1}{4^n\mu_n}I\cap [2^{-n},2^n]=&\{2^j:-n\le j\le n\};\\
\label{eee3}\frac{1}{8^n\mu_n}I\cap [2^{-n},2^n]=&\{2^j:-n\le j\le 0\}.
\end{align}
Accordingly, let $\omega$ be a nonprincipal ultrafilter. If $\omega$ contains $\{\kappa^n\mu_n:n\ge 1\}$ for some $\kappa=2,4,8$, then 
\[\frac{1}{\kappa^n\mu_n}I\cap [2^{-n},2^n]\subset \{2^j:-n\le j\le n\};\]
and thus $I$ is $\omega$-doubling. In view of Lemma \ref{cp_bound2}, we furthermore obtain

\begin{itemize}
\item If $\omega$ contains $\{2^n\mu_n:n\ge 1\}$, using (\ref{eee1})
we see that $\Cone_\omega(I)=\{0\}\cup\{2^n:n\in\mathbf{N}\cup\{0\}\}$
and thus $I$ satisfies the conditions of (\ref{al2}); 

\item if $\omega$ contains $\{4^n\mu_n:n\ge 1\}$, using (\ref{eee2}), 
we see that $\Cone_\omega(I)=\{0\}\cup\{2^n:n\in\mathbf{Z}\}$
and thus $I$ satisfies the conditions of (\ref{al5});
\item if $\omega$ contains $\{8^n\mu_n:n\ge 1\}$, using (\ref{eee3}), we see that $\Cone_\omega(I)=\{0\}\cup\{2^n:n\in -\mathbf{N}\cup\{0\}\}$
and thus $I$ satisfies the conditions of (\ref{al4}).
\end{itemize}

Now set $L_n=\{2^n,2^n+1,\dots,8^n-1,8^n\}$ and $\lambda_n=m(L_n,n)$.
Dividing by $2^n$, $4^n$  we get
\begin{align}
\label{fee1}\frac{1}{2^n\lambda_n}I\cap [2^{-n},2^n]=&\{1,1+2^{-n},1+2.2^{-n},1+3.2^{-n},\dots,2^n-2^{-n},2^n\};\\
\label{fee3}\frac{1}{4^n\lambda_n}I\cap [2^{-n},2^n]=&\\
\notag \{2^{-n},\dots, 1-4^{-n}&,  \;1 \;,1+4^{-n},1+2.4^{-n},\dots ,2^n-4^{-n},2^n\}.
\end{align}
We deduce:

\begin{itemize}
\item If $\omega$ contains $\{2^n\lambda_n:n\ge 1\}$ then $\Cone_\omega(I)=\{0\}\cup [1,\infty\mathclose[$. We see that $J$ is $\omega$-discrete but not $\omega$-doubling as in (\ref{al3}).

\item if $\omega$ contains $\{4^n\lambda_n:n\ge 1\}$ then $\Cone_\omega(I)=[0,+\infty\mathclose[$ and $J$ is not lower $\omega$-semidoubling as in (\ref{al7}).
\end{itemize}

For the missing case, set $M_n=\{2^j:n\le j\le 2n\}\cup\{4^n,4^n+1,\dots 8^n-1,8^n\}$ and $\nu_n=m(M_n,n)$. We have 
\[\frac{1}{4^n\lambda_n}I\cap [2^{-n},2^n]=\{2^j:-2n\le j\le 0\}\cup \{1,1+4^{-n},1+2.4^{-n},\dots\}\]
Then if $\omega$ contains $\{4^n\nu_n:n\ge 1\}$, then $\Cone_\omega(I)=\{0\}\cup\{2^{-n}:n\in \mathbf{N}\}\cup [1,\infty\mathclose[$; we see that $I$ is not $\omega$-discrete, not $\omega$-doubling but is lower $\omega$-semidoubling as in (\ref{al6}).
\end{proof}

\subsection{Auxiliary lemmas}

This part contains some more basic lemmas that will be used in the sequel.

\subsubsection{Ultraproducts and cones}

\begin{lem}\label{cardco}Let $X$ be a metric space and $x_0$ a base-point. Let $c_n$ be the minimal number of closed $n$-balls needed to cover the closed $2n$-ball around $x_0$. If $\omega$ is a nonprincipal ultrafilter and $\lim_\omega c_n=\infty$, then $\Cone_\omega(X)$  has continuum cardinality.
\end{lem}
\begin{proof}
Let $C_n$ be a set of $c_n$ points in the $2n$-ball, at mutual distance at least $n$. There is a natural map from the ultraproduct $\prod^\omega C_n$ into $\Cone_\omega(X)$, which maps any two distinct elements to points at distance at least 1. If we show that $\prod^\omega C_n$ has continuum cardinality, we are done. To check the latter (which is well-known), since it is no longer related to the original distance on $C_n$, we can now identify $C_n$ with the subset $\{0,1/c_n,\dots,(c_n-1)/c_n\}$ of $[0,1]$, and map any $u\in\prod^\omega C_n$ to $\phi(u)=\lim_\omega u_n\in[0,1]$. Then this map $\prod^\omega C_n\to [0,1]$ is easily shown to be surjective.
\end{proof}

\begin{lem}\label{upca}
If $X$ is a metric space with a countable cobounded subset (e.g.\ $X$ is separable), then $\Cone_\omega(X)$ and $\pi_1(\Cone_\omega(X))$ have cardinality at most continuum.
\end{lem}
\begin{proof}
If $D$ is a countable cobounded subset, every element of $\Cone_\omega(X)$ is determined by a sequence in $D$, whence the conclusion. Also, every loop in $\Cone_\omega(X)$ is determined by a continuous map from a dense subset $C$ of the circle to $\Cone_\omega(X)$. Since $(2^{\aleph_0})^{\aleph_0}=2^{\aleph_0}$, we are done.
\end{proof}

\subsubsection{Topological groups}

\begin{lem}\label{projl}
Let $G$ be a group endowed with a biinvariant complete Hausdorff ultrametric distance; denote by $G_r$ its closed $r$-ball (which is an open normal subgroup). Then the natural continuous homomorphism from $G$ to the projective limit ${\underleftarrow{\lim}}_{r\to 0}G/G_r$ is a topological isomorphism. In particular, if $G_r$ has finite index for all $r>0$, then $G$ is compact.
\end{lem}
\begin{proof}
Note that the fact that the distance on $G$ is ultrametric implies that $G_r$ is a subgroup, and its bi-invariance implies that $G_r$ is normal. The projective limit carries a natural bi-invariant ultrametric such that the closed ball of radius $r$ consists of $(g_s)_{s>0}$ such that $g_s=1_{G/G_s}$ for all $s>r$.  
This distance makes the natural continuous homomorphism $\psi$ from $G$ to the projective limit an isometry onto its image. But since the latter is dense and $G$ is complete, $\psi$ is an  isomorphism of metric groups.  
\end{proof}

If $\kappa$ is a cardinal, we denote by $\F_p^{(\kappa)}$ the discrete group defined as the direct sum (or restricted direct product) of $\kappa$ copies of the cyclic group $\F_p$, and the full product $\F_p^\kappa$ is endowed with the product topology, which makes it a compact group. 

\begin{lem}\label{expp2}
Let $G$ be a locally compact abelian group of exponent $p$. Then there exist cardinals $\kappa_0,\kappa_1$ such that $G\simeq \F_p^{\kappa_0}\times \F_p^{(\kappa_1)}$. 
\end{lem}
\begin{proof}
If $G$ is discrete, it is a vector space over $\F_p$, so if $\kappa$ is the cardinal of a basis then $G\simeq\F_p^{(\kappa)}$.

If $G$ is compact, its Pontryagin dual $\widehat{G}=\textnormal{Hom}(G,\mathbf{R}/\mathbf{Z})$ is a discrete locally compact abelian group of exponent $p$, so $\widehat{G}$ is isomorphic to $\F_p^{(\kappa_0)}$ for some cardinal $\kappa$, by the discrete case. Hence we have
\[\widehat{\widehat{G\,\!}}\simeq\textnormal{Hom}(\F_p^{(\kappa)},\mathbf{R}/\mathbf{Z})\simeq \textnormal{Hom}(\F_p,\mathbf{R}/\mathbf{Z})^\kappa\simeq\F_p^\kappa.\]
Pontryagin duality implies that $G$ is topologically isomorphic to $\widehat{\widehat{G}}$, so $G\simeq\F_p^\kappa$.

Now let $G$ be arbitrary (as in the statement of the lemma). Since $G$ is a locally compact abelian group with no closed subgroup isomorphic to $\mathbf{R}$, it follows from \cite[Theorem 9.8]{HR} that $G$ has an open compact subgroup $K$. Forgetting the topology and viewing $G$ as a vector space over $\F_p$, there exists a supplement subspace $D$, so that $G\simeq K\times D$ as a abstract group. Since $K$ is open, $D$ is discrete and therefore $G\simeq K\times D$ as a topological group. By the discrete and compact cases, we can write $D\simeq\F_p^{(\kappa_1)}$ and $K\simeq\F_p^{\kappa_0}$ for some cardinals $\kappa_0,\kappa_1$. This finishes the proof.
\end{proof}

\subsection{Cones of  $\F_p^{(J)}$}\label{fpj}

If $X$ is a metric space, define 
\[D(X)=\{d(x,y):\,(x,y)\in X^2\};\quad D_{\log}(X)=\big\{\log(t):\,t\in D(X)\smallsetminus\{0\}\big\}.\]
If $X,Y$ are bilipschitz homeomorphic metric spaces, then $D_{\log}(X)$ and $D_{\log}(Y)$ are at finite Hausdorff distance (as defined before Lemma \ref{bilcon}). It is straightforward that if $X$ is a nonempty metric space with transitive isometry group, we have \[D(\Cone_\omega(X))=\Cone_\omega(D(X)).\]

Let $p$ be a prime and for $J\subset\mathbf{N}$ let $\F_p^{(J)}$ be the metric group introduced at the beginning of this Section \ref{csds}.
Observe that for $J\subset\mathbf{N}$, we have $D(\F_p^{(J)})=\{0\}\cup J$. In particular, \[D(\Cone_\omega(\F_p^{(J)}))=\Cone_\omega(J).\]

If particular, whether $\Cone_\omega(\F_p^{(J)})$ is trivial, bounded, or discrete are reflected in $\Cone_\omega(J)$. Thus from Lemmas \ref{critcone2} and \ref{cp_bound2} we obtain

\begin{lem}\label{cotr}
For every nonempty subset $J$ of $\mathbf{N}$ and nonprincipal ultrafilter $\omega$ we have
\begin{itemize}
\item $\Cone_\omega\big(\F_p^{(J)}\big)=\{0\}$ $\Leftrightarrow$ $\Cone_\omega(J)=\{0\}$ $\Leftrightarrow$ $\lim_\omega\sigma_J=\infty$;
\item $\Cone_\omega\big(\F_p^{(J)}\big)$ is bounded $\Leftrightarrow$ $\Cone_\omega(J)$ is bounded $\Leftrightarrow$ $J$ is $\omega$-bounded;
\item $\Cone_\omega\big(\F_p^{(J)}\big)$ is discrete $\Leftrightarrow$ 0 is isolated in $\Cone_\omega(J)$ $\Leftrightarrow$ $J$ is $\omega$-discrete.
\end{itemize}
\end{lem}

\begin{proof}[Proof of Corollary \ref{contifg}]
Let $\mathcal{U}$ be the set of all nonprincipal ultrafilters on $\mathbf{N}$. We define a quasi-isometry invariant $\scu(X)$ of a metric space $X$ as the subset of $\mathcal{U}$ consisting of those $\omega$ such that $\Cone_\omega(X)$ is simply connected. By (\ref{identif}), if $J$ is a nonempty subset of $\mathbf{N}$ then by Lemma \ref{cotr}
\begin{align*}\scu(\Gamma_J)= & \{\omega\in\mathcal{U}:\Cone_\omega\big(\F_p^{(J)}\big)=\{0\}\}\\
= & \{\omega\in\mathcal{U}:\Cone_\omega(J)=\{0\}\}.\end{align*}
So by Proposition \ref{pleino2}, there are continuum many sets of nonprincipal ultrafilters on $\mathbf{N}$ that can be obtained as $\scu(\Gamma_J)$ for some $J$ and thus these are continuum many non-quasi-isometric $\Gamma_J$.
\end{proof}

\begin{proof}[Proof of Corollary \ref{resomany}]
Since for every $J$ we have 
\[D_{\log}\big(\Cone_\omega\big(\F_p^{(J)}\big)\big)=\log\Big(\Cone_\omega\big(\F_p^{(J)}\big)\smallsetminus\{0\}\Big),\] 
Lemma \ref{bilcon} states that for some $K\subset\mathbf{N}$ the subsets
$D_{\log}\big(\Cone_\omega\big(\F_p^{(K)}\big)\big)$ achieve continuum many subsets of $\mathbf{R}$ at pairwise infinite Hausdorff distance when $\omega$ ranges over nonprincipal ultrafilters; thus $\Cone_\omega\big(\F_p^{(K)}\big)$ achieves continuum many bilipschitz non-equivalent metric spaces and hence by (\ref{identif}), $\pi_1(\Cone_\omega(\Gamma_K))$ achieves continuum many non-bilipschitz spaces. In particular, $\Cone_\omega(\Gamma_K)$ achieves continuum many non-bilipschitz spaces when $\omega$ ranges over nonprincipal ultrafilters.
\end{proof}

Other metric properties of $\Cone_\omega\big(\F_p^{(J)})\big)$ can be characterized in terms of $J$.

\begin{prop}\label{doubca}
For every subset $J$ of $\mathbf{N}$ and nonprincipal ultrafilter $\omega$ we have
\begin{enumerate}
\item\label{iproper} $\Cone_\omega\big(\F_p^{(J)}\big)$ is proper (i.e.\ its closed bounded subsets are compact) if and only if $J$ is $\omega$-doubling; otherwise it contains a closed bounded discrete subgroup of continuum cardinality;
\item\label{ilc} $\Cone_\omega\big(\F_p^{(J)}\big)$ is locally compact if and only if $J$ is lower $\omega$-semidoubling;
\item\label{isb} $\Cone_\omega\big(\F_p^{(J)}\big)$ is $\sigma$-bounded (i.e.\ a countable union of bounded subsets) if and only if $J$ is upper $\omega$-semidoubling; otherwise every ball has continuum index in $\Cone_\omega\big(\F_p^{(J)}\big)$.
\end{enumerate}
\end{prop}

Denote by $G_r$ (resp.\ $G^\circ_r$) the closed (resp.\ open) $r$-ball in $\Cone_\omega\big(\F_p^{(J)}\big)$; both are open subgroups. We begin by the following lemma.

\begin{lem}\label{jc}
Let $J$ be a subset of $\mathbf{N}$. Fix $0<c<C<\infty$. Then we have
\begin{enumerate}
\item\label{jc_inf} 
If $\lim_\omega\#([cn,Cn]\cap J)=\infty$, then there exist continuum many points in $G_C$ at pairwise distance $\ge c$; in particular the open subgroup $G^\circ_c$ has continuum index and $G_C$ is not compact.
\item\label{jc_fin} If $\lim_\omega\#([cn,Cn]\cap J)<\infty$ then $G^\circ_C/G_c$ is finite.  
\end{enumerate}
\end{lem}
\begin{proof}
Write $J_n=[cn,Cn]\cap J$. Then there is an obvious homomorphism $\phi$ from the ultraproduct $\prod^\omega\F_p^{(J_n)}$ to the closed $C$-ball $G_C$ in $\Cone_\omega\big(\F_p^{(\mathbf{N})}\big)$, mapping any two distinct elements to elements at distance $\ge c$.

Let us first prove (\ref{jc_inf}). Since $\lim_\omega\#J_n=\infty$, the ultraproduct $\prod^\omega \F_p^{(J_n)}$ has continuum cardinality (see e.g.\ the proof of Lemma \ref{cardco}). Thus the image of $\phi$ satisfies the required property.

Let us now prove (\ref{jc_fin}). Consider the composite homomorphism \[{\prod}^\omega\F_p^{(J_n)}\stackrel{\phi}\to G_C\to G_C/G_c\,.\] We claim that its image contains $G^\circ_C/G_c$. Indeed, let $(u_n)$ be in $G^\circ_C$; we can suppose that $|u_n|\le Cn$ for all $n$. Set $K_n= J\cap [0,cn\mathclose[$. Write $u_n=v_n+w_n$ with $v_n\in \F_p^{(J_n)}$ and $w_n\in \F_p^{(K_n)}$. Then $|w_n|\le cn$, so $(u_n)=(v_n)$ modulo $G_c$, while clearly $(v_n)$ is in the image of $\phi$. Since $\lim_\omega\#J_n=\infty$, the ultraproduct $\prod^\omega \F_p^{(J_n)}$ is finite. We deduce that $G_C/G^\circ_c$ is finite.
\end{proof}

\begin{proof}[Proof of Proposition \ref{doubca}]
Suppose that $J$ is $\omega$-doubling. Then for all $0<c<C<\infty$ we have $\lim_\omega\#([cn,Cn]\cap J)<\infty$. By Lemma \ref{jc}(\ref{jc_fin}), we deduce that for all $0<c<C<\infty$ we have $G^\circ_C/G_c$ finite. In particular, $G_C/G_c$ is finite (because it is contained in $G^\circ_{2C}/G_c$). It follows from Lemma \ref{projl} that $G_C$ is compact. Since this holds for all $C$, we deduce that $G$ is proper.

If $J$ is lower $\omega$-semidoubling, the same reasoning works for $C$ small enough and we deduce that $G_C$ is compact for small $C$. Thus $G$ is locally compact.

If $J$ is upper $\omega$-semidoubling, the we obtain that $G^\circ_C/G_c$ is finite for a given $c>0$ and all $C>c$. It follows that $G^\circ_C$ is a finite union of $c$-balls for all $C>c$ (namely cosets of $G_c$), and thus that $G$ is a countable union of bounded subsets.

Conversely if $G$ is proper, then for all $C$, the ball $G_C$ is compact and by the contraposition of Lemma \ref{jc}(\ref{jc_fin}) we deduce that $\lim_\omega\#([cn,Cn]\cap J)<\infty$ for all $c<C$. Thus $J$ is $\omega$-doubling.

If $G$ is locally compact, this holds for some given $C$ and all $c<C$ and this shows similarly that $J$ is lower $\omega$-semidoubling.

If $G$ is $\sigma$-bounded, then there exists $c$ such that $G$ is covered by countably many open $c$-balls; hence the index of $G^\circ_c$ in $G_C$ is at most countable for all $C>c$. The contraposition of Lemma \ref{jc}(\ref{jc_fin}) shows that $\lim_\omega\#([cn,Cn]\cap J)<\infty$ for all $C>c$. This means that $J$ is upper $\omega$-semidoubling.
\end{proof}

In the same vein, we have

\begin{prop}\label{widf}
Let $J$ be any subset of $\mathbf{N}$ and $\omega$ a nonprincipal ultrafilter. Then $\Cone_\omega\big(\F_p^{(J)}\big)$ is finite if and only if $J$ has finite $\omega$-width, which is then equal to the dimension of 
$\Cone_\omega\big(\F_p^{(J)}\big)$ over $\F_p$.
\end{prop}
\begin{proof}
For a nonnegative integer $k$, suppose that the $\omega$-width of $J$ is $\ge k$. Then for some $C>1$ and $M\in\omega$, for all $n\in M$ we have $\#(M\cap [C^{-1}n,Cn])\ge k$. Pick $k$ distinct elements $v_{n,1}<\dots<v_{n,k}$ in this set. For $1\le i\le k$, let $u_i$ be the element of $\Cone_\omega\big(\F_p^{(J)}\big)$ defined by the sequence $u_i(n)=\delta_{v_{n,i}}$ if $n\in M$ and $u_i(n)=0$ otherwise. Then it is readily seen that $(u_i)_{1\le i\le k}$ is $\F_p$-free in $\Cone_\omega\big(\F_p^{(J)}\big)$ and thus the dimension of the latter is $\ge k$.

Let us now suppose that the $\omega$-width of $J$ is exactly equal to $k$. There exists $M\in\omega$ and $C_0>1$ such that for every $C\ge C_0$ we have $\#(M\cap [C^{-1}n,Cn])=k$. So for $C\ge C_0$ the elements $u_i$ defined above are uniquely defined and do not depend on $C$. Let $x=(x_n)$ be an element in $\Cone_\omega\big(\F_p^{(J)}\big)$ and let us show that $x$ belongs to the subspace generated by the $u_i$, showing that the dimension of the cone is $\le k$ and finishing the proof. 

For some $C_1\ge C_0$ we have $x_n\le C_1n$ for all $n\in M$. Fix any $C\ge C_1$. For $n\in M$, write $x_n=w_n+e_n$ where $w_n$ is an $\F_p$-combination of the $v_{n,i}$ and $e_n$ has support in $[0,C^{-1}n\mathclose[$; for $n\notin M$ just define $w_n=x_n$ and $e_n=0$. This writes $x=w+e$, where $w$ is an $\F_p$-combination of the $u_i$ and the length of $e$ is $\le C^{-1}$. This shows that the distance of $x$ with the vector space $V$ generated by the $u_i$ is $\le C^{-1}$; since $C$ can be taken arbitrary large, this distance is zero; since $V$ is finite hence closed, this shows that $x\in V$.
\end{proof}

\begin{proof}[Proof of Corollary \ref{littlepi}]
Set $K=\{2^{2^n}:n\in\mathbf{N}\}$. Then it follows from Proposition \ref{widf} that $\Cone_\omega\big(\F_p^{(K)}\big)$ has rank at most~1 for every $\omega$, and that if moreover $\omega$ contains $K$, then it has rank~1. Therefore, using (\ref{identif}), the statement of the corollary holds for the group $\Gamma_K$. 
\end{proof}

We now classify the groups $\Cone_\omega\big(\F_p^{(J)}\big)$ up to isomorphism of topological groups. Denote by $\simeq$ the isomorphy relation within topological groups. 
For any topological group $H$, we denote by $H^{\kappa}$ the product of $\kappa$ copies of $H$, with the product topology; this topology is compact when $H$ is compact.
Let $\epsilon_0$ be the infinite countable cardinal, and $\epsilon_1=2^{\epsilon_0}$ be the continuum cardinal.

\begin{thm}\label{thm:cla}
Let $J\subset \mathbf{N}$ be a subset and $G=\Cone_\omega\big(\F_p^{(J)}\big)$. Then as a topological group, we can describe $G$ by the following discussion:
\begin{itemize}
\item If $G$ is discrete, then
 \begin{itemize}
 \item if $G$ is bounded and proper, $G\simeq \F_p^n$ for some finite $n$;
 \item if $G$ is unbounded and proper, $G\simeq \F_p^{(\epsilon_0)}$
 \item if $G$ is not proper, $G\simeq \F_p^{(\epsilon_1)}$;
 \end{itemize}
\item If $G$ is locally compact and not discrete, then
 \begin{itemize}
 \item if $G$ is compact (i.e.\ bounded and proper), $G\simeq \F_p^{\epsilon_0}$;
 \item if $G$ is unbounded and proper, $G\simeq \F_p^{\epsilon_0}\times \F_p^{(\epsilon_0)}$
 \item if $G$ is not proper, $G\simeq \F_p^{\epsilon_0}\times \F_p^{(\epsilon_1)}$;
 \end{itemize}
\item If $G$ is not locally compact, then $G\simeq \big(\F_p^{(\epsilon_1)}\big)^{\epsilon_0}$.
\end{itemize}
\end{thm}
The above metric properties of $G$ are characterized in Lemma \ref{cotr}, Propositions \ref{doubca} and \ref{widf}. Proposition \ref{allk} shows that all cases can be achieved, and can even be achieved by a single subset $K$ for different ultrafilters.

\begin{proof}
Note that $G$ is abelian of exponent $p$, and has cardinality $\le\epsilon_1=2^{\aleph_0}$ by Lemma \ref{upca}.

If $G$ is locally compact, there exist, by Lemma \ref{expp2}, cardinals $\kappa_0$ and $\kappa_1$ such that $G\simeq \F_p^{\kappa_0}\times \F_p^{(\kappa_1)}$. Of course we can suppose that if exactly one among $\kappa_0$ and $\kappa_1$ is infinite, then the other is 0.

Clearly, if $G$ is compact then $\kappa_1$ is finite and if $G$ is proper and not compact, then $\kappa_1=\epsilon_0$. Otherwise $G$ is not proper; 
Proposition \ref{doubca}(\ref{iproper}) then implies $\kappa_1=\epsilon_1$. 

In the same fashion, if $G$ is discrete then $\kappa_0$ is finite. Otherwise $G$ is not discrete and still assuming $G$ locally compact, some ball $G_R$ of $G$ is compact. By Lemma \ref{projl}, $G_R$ is the projective limit of the discrete groups $G_R/G_{2^{-n}R}$ when $n\to\infty$. Compactness implies that $G_R/G_{2^{-n}R}$ is finite. So $G_R$ is isomorphic to $\F_p^{\epsilon_0}$. It follows that $\kappa_0=\epsilon_0$ if $G$ is not discrete and locally compact. 

Combining all these observations, we obtain the full statement of the theorem in case $G$ is locally compact.

Finally assume that $G$ is not locally compact. By Proposition \ref{doubca}(\ref{ilc}), for every $0<C\le\infty$ there exists $0<c<C$ such that $G_C/G_c$ has continuum cardinality. Hence define a decreasing sequence $(\eps_n)$, tending to zero, such that $G_{\eps_n}/G_{\eps_{n+1}}$ has continuum cardinality for all $n$ (and $\eps_0=\infty$). So, again using Lemma \ref{projl}, $G$ is the projective limit of all $G/G_{\eps_n}$. Let $L_n$ be a complement subgroup of $G_{\eps_n}$ in $G_{\eps_{n-1}}$ (where $G_0=G$). So
$$G=G_{\eps_0}=L_1\oplus G_{\eps_1}=L_1\oplus L_2\oplus G_{\eps_2}=\dots=\bigoplus_{i=1}^kL_i\oplus G_{\eps_k}=\dots$$
This provides a homomorphism $\iota:G\to\prod L_i$, clearly injective. Also map a sequence $\ell=(\ell_i)$ in $\prod L_i$ to the sequence $\rho(\ell)=(\prod_{i=1}^k\ell_i)_k$, which belongs to the projective limit. This $\rho$ is a continuous homomorphism and $\iota\circ\rho$ is the identity, so the injectivity of $\iota$ implies that $\iota$ is a topological isomorphism. Since each $L_i$ is isomorphic to $\F_p^{(\epsilon_1)}$, we are done.
\end{proof}

\begin{cor}\label{carpc}
Let $\mathbf{K}$ be a nondiscrete locally compact field of characteristic $p>0$ and $|\cdot|$ an absolute value on $\mathbf{K}$. Then the group $G=\Cone_\omega(\mathbf{K},\log(1+|\cdot|))$ is topologically isomorphic to $(\F_p^{(\epsilon_1)})^{\epsilon_0}$. 
\end{cor}
\begin{proof}
We see that $(\mathbf{K},\log(1+|\cdot|))$ is quasi-isometric to $\mathbf{F}_p^{(\mathbf{N})}$ endowed the distance introduced at the beginning of this Section \ref{csds}. This yields the result.
\end{proof}

In a similar fashion, we have, in characteristic zero, the following proposition.

\begin{prop}\label{car0c}
Let $\mathbf{K}$ be a nondiscrete locally compact field of characteristic zero and $|\cdot|$ an absolute value on $\mathbf{K}$. Then the group $G=\Cone_\omega(\mathbf{K},\log(1+|\cdot|))$ is topologically isomorphic to $(\mathbf{Q}^{(\epsilon_1)})^{\epsilon_0}$ (where $\epsilon_1=2^{\epsilon_0}$ and $\epsilon_0$ is countable).
\end{prop}
\begin{proof}
First observe that $G$ is naturally a vector space over $\mathbf{Q}$.
The metric defined by $\log(1+|\cdot|)$ on $\mathbf{K}$ is ultrametric in the non-Archimedean case, and in all cases satisfies the quasi-ultrametric condition $$d(x,z)\le\max(d(x,y),d(y,z))+\log(2),$$ so $G$ is ultrametric for all $\mathbf{K}$; it is readily observed that the index between any two balls $G_r\subset G_R$ of distinct radii is continuum, so that the quotient $G_R/G_r$ is isomorphic to $\mathbf{Q}^{(\epsilon_1)}$. So the proof of the non-locally-compact part of Theorem \ref{thm:cla} shows, without change, that $G$ is isomorphic to $(\mathbf{Q}^{(\epsilon_1)})^{\epsilon_0}$.
\end{proof}

\end{document}